\documentclass[12pt]{amsproc}

\usepackage{amsfonts}
\usepackage{amsmath}
\usepackage{amssymb}

\usepackage{mathrsfs}
\usepackage{amscd}
\usepackage[all]{xy}
\usepackage[pdftex,final]{graphicx}

\usepackage{hyperref}

\usepackage{exscale,txfonts}

\newtheorem{lem}{Lemma}[section]
\newtheorem{thm}[lem]{Theorem}
\newtheorem{prop}[lem]{Proposition}
\newtheorem{cor}[lem]{Corollary}
\theoremstyle{definition}

\newtheorem{exa}[lem]{Example}
\newtheorem{rem}[lem]{Remark}

\newcommand{\Q}{\Bbb{Q}}
\newcommand{\F}[1]{\Bbb{F}_{#1}}

\newcommand{\Z}{\Bbb{Z}}

\newcommand{\R}{\Bbb{R}}

\newcommand{\K}{\mathcal{K}}

\newcommand{\G}{{\Gamma}}

\newcommand{\spl}[2]{\mathrm{SL}_{#1}(#2)}

\newcommand{\gl}[2]{\mathrm{GL}_{#1}(#2)}

\newcommand{\pb}[1]{\mathcal{P}(#1)}
\newcommand{\qpb}[1]{\widetilde{\mathcal{P}}(#1)}
\newcommand{\redpb}[1]{\overline{\mathcal{P}}(#1)}
\newcommand{\rredpb}[1]{\overline{\mathcal{P}}(#1)}
\newcommand{\rpb}[1]{{\mathcal{RP}}(#1)}
\newcommand{\rpbker}[1]{\mathcal{RP}_1(#1)}
\newcommand{\rrpb}[1]{\overline{\mathcal{RP}}(#1)}
\newcommand{\qrpb}[1]{\widetilde{\mathcal{RP}}(#1)}
\newcommand{\rrpbker}[1]{\overline{\mathcal{RP}}_1(#1)}
\newcommand{\rrrpb}[1]{\widehat{\rpb{#1}}}
\newcommand{\qrpbker}[1]{\widetilde{\mathcal{RP}}_1(#1)}

\newcommand{\rsc}[1]{\mathcal{RP}'(#1)}
\newcommand{\grsc}[1]{[ #1 ]'}

\newcommand{\bl}[1]{\mathcal{B}(#1)}

\newcommand{\qbl}[1]{\widetilde{\mathcal{B}}(#1)}

\newcommand{\rbl}[1]{{\mathcal{RB}}(#1)}
\newcommand{\rblker}[1]{\mathcal{RB}_0(#1)}%
\newcommand{\rrbl}[1]{\overline{\mathcal{RB}}(#1)}
\newcommand{\qrbl}[1]{\widetilde{\mathcal{RB}}(#1)}

\newcommand{\gpb}[1]{\left[ #1\right]} 
\newcommand{\sgpb}[1]{[ #1]} 
\newcommand{\sus}[1]{\left\{ #1\right\} }
\newcommand{\suss}[2]{\psi_{#1}\left( #2\right)}

\newcommand{\indf}[3]{\ind{#2}{#1}{#3}}

\newcommand{\ks}[2]{{\K}^{\mbox{\tiny $(#1)$}}_{#2}}

\newcommand{\kv}[1]{\mathcal{L}_{#1}}
\newcommand{\pbconst}[1]{\tilde{C}_{#1}}

\newcommand{\bconst}[1]{C_{#1}}
\newcommand{\cconst}[1]{D_{#1}}
\newcommand{\cconstmod}[1]{\mathcal{D}_{#1}}
\newcommand{\icconstmod}[1]{\mathcal{D}^0_{#1}}

\newcommand{\bw}{\Lambda}
\newcommand{\rbw}{\widetilde{\bw}}

\newcommand{\spec}[1]{S_{#1}}

\newcommand{\sq}[1]{{#1}^\times/({#1}^\times)^2}
\newcommand{\tnorm}[1]{\mathcal{N}_{#1}}
\newcommand{\rsq}[1]{{\mathcal{G}}_{#1}}

\newcommand{\chara}[1]{\mathrm{char}(#1)}

\newcommand{\laur}[2]{{#1}[{#2},{#2}^{-1}]}

\renewcommand{\dim}[2]{\mathrm{dim}_{#1}\left(#2\right)}

\newcommand{\card}[1]{\left| #1 \right|}

\renewcommand{\forall}{\mbox{ for all }}

\newcommand{\dual}[1]{\widehat{#1}}

\newcommand{\id}[1]{\mathrm{Id}_{#1}}

\renewcommand{\ker}[1]{\mathrm{Ker}(#1)}

\newcommand{\coker}[1]{\mathrm{Coker}(#1)}

\newcommand{\hmz}[2]{\mathrm{Hom}(#1,#2)}

\newcommand{\unitr}[1]{{#1}^\times}

\newcommand{\sgr}[1]{\mathrm{R}_{#1}}
\newcommand{\rsgr}[1]{\widehat{\sgr{#1}}} 
\newcommand{\an}[1]{\left\langle{#1}\right\rangle}
\newcommand{\pf}[1]{\left\langle\!\left\langle{#1}\right\rangle\!\right\rangle}
\newcommand{\aug}[1]{\mathcal{I}_{#1}}

\newcommand{\laurs}[2]{{#1}\left(\!\left( #2 \right)\!\right)}

\newcommand{\igr}[1]{\Z[#1]}

\newcommand{\pp}[1]{\mathrm{p}_{#1}^+}

\newcommand{\ep}[1]{\mathrm{e}_{#1}^+} 
\newcommand{\epm}[1]{\mathrm{e}_{#1}^-} 
\newcommand{\idem}[1]{\mathrm{e}_{#1}}
\newcommand{\idm}[2]{\mathrm{e}^{#1}_{#2}}

\newcommand{\upp}[2]{{#1}^{(#2)}}
\newcommand{\dwn}[2]{{#1}_{(#2)}}

\newcommand{\pieps}[3]{{#1}[#2,#3]}

\newcommand{\asymm}{\circ}
\newcommand{\asym}[3]{\mathrm{S}^{#1}_{#2}(#3)}
\newcommand{\qasym}[3]{\tilde{\mathrm{S}}^{#1}_{#2}(#3)}

\newcommand{\symb}[1]{\Sigma\left( #1\right)}

\newcommand{\zhalf}[1]{{#1}\left[ \tfrac{1}{2}\right]}
\newcommand{\znth}[2]{{#1}\left[ \tfrac{1}{#2}\right]}

\newcommand{\res}[1]{\arrowvert_{#1}}

\newcommand{\mwk}[2]{K^{\mathrm{\small MW}}_{#1}({#2})}

\newcommand{\kind}[1]{K^{\mathrm{\small ind}}_3(#1)}

\newcommand{\ho}[3]{\mathrm{H}_{#1}(#2,#3 )}
\newcommand{\hot}[2]{\mathrm{H}_3\left( \spl{2}{#1},#2\right)_0}
\newcommand{\hoz}[2]{\ho{#1}{#2}{\Z}}

\newcommand{\ind}[2]{\mathrm{Ind}^{#1}_{#2}}

\newcommand{\Tor}[2]{\mathrm{Tor}^{\Z}_{1}(#1,#2)}

\title{
The third homology of $\mathrm{SL}_2$ of  fields with discrete valuation
}
\author{Kevin Hutchinson}

\address{School of Mathematics and Statistics,
 University College Dublin}
\email{kevin.hutchinson@ucd.ie}
\date{\today}


\begin{document}
\bibliographystyle{plain}
\maketitle

\begin{abstract}
For a field $F$ with discrete valuation and residue field $k$ 
we relate the third homology of $\spl{2}{F}$ 
with half-integral coefficients to the third homology of $\spl{2}{k}$ and a certain refined 
scissors congruence group of $k$. As an application, we obtain explicit calculations of the 
third homology of $\mathrm{SL}_2$ of certain higher-dimensional local fields in terms of 
scissors congruence groups, in the sense of Dupont and Sah, of the residue fields of the 
associated valuations.  
\end{abstract}

\section{Introduction}\label{sec:intro}
\subsection{Overview}
Let $F$ be a field. There is a natural surjective map from $\hoz{3}{\spl{2}{F}}$ to the 
indecomposable quotient, $\kind{F}$, of the third Quillen $K$-group of $F$. 
This article concerns the properties and calculation of the kernel of this homomorphism, 
which we denote $\hot{F}{\Z}$.

Our calculations below and elsewhere 
show that in general $\hot{F}{\Z}$ may be  large, at least
 when $F^\times \not= (F^\times)^2$.  For example, for $F=\Q$ 
the torsion group 
$\hot{\Q}{\Z}$ maps onto an infinite direct sum of non-trivial finite cyclic groups 
(\cite{hut:rbl11}). For $n\geq 3$ by way of contrast, 
$\ho{3}{\spl{n}{\Q}}{\Z}\cong \kind{\Q}\cong \Z/24$.

However, our purpose is to show that 
the functor $\hot{F}{\Z}$ has good properties which make it amenable to calculation. 
Because of systematic complications involving $2$-torsion, however, 
we will state our main results for $\zhalf{\Z}$-modules in this article. As a typical example, 
for any prime $p$ we show  that 
\[
\ho{3}{\spl{2}{\laurs{\Q_p}{x}}}{\zhalf{\Z}}\cong \zhalf{\kind{\laurs{\Q_p}{x}}}\oplus 
\zhalf{\pb{\Q_p}}\oplus \zhalf{\pb{\F{p}}}^{\oplus 2}
\]
where $\pb{F}$ is the scissors congruence group of the field $F$ (see Example \ref{exa:qpxx}). 

Our approach exploits the close relationship of $\hot{F}{\Z}$ to certain 
refined or generalised \emph{scissors congruence groups}, which are given by explicit 
presentations. 

The scissors congruence group, $\pb{F}$, 
of a field $F$ was introduced by Dupont and Sah in their study of Hilbert's third problem in 
hyperbolic $3$-space (Dupont and Sah, \cite{sah:dupont}). It is an abelian group given by 
an explicit presentation (see section \ref{sec:bloch} below for details) and is closely connected 
with $\kind{F}$ (\cite{sus:bloch}). The \emph{refined scissors congruence group} (see 
\cite{hut:rbl11} and \cite{hut:laurent}), $\rpbker{F}$ 
is connected in an analogous way to the group $\ho{3}{\spl{2}{F}}{\zhalf{\Z}}$. 

There is 
a natural homomorphism $\rpbker{F}\to \pb{F}$ which is an isomorphism when the field $F$ is 
quadratically closed. For a general field $F$ there is a natural short exact sequence
\[
0\to \hot{F}{\zhalf{\Z}}\to \zhalf{\rpbker{F}}\to \zhalf{\pb{F}}\to 0. 
\]

The group $\rpbker{F}$ itself occurs naturally in the calculation of the third homology 
of $\mathrm{SL}_2$ of rings: In \cite{hut:laurent} we showed that for any 
infinite field $F$ there is a natural split 
short exact sequence 
\[
0\to \ho{3}{\spl{2}{F}}{\zhalf{\Z}}\to \ho{3}{\spl{2}{\laur{F}{t}}}{\zhalf{\Z}}\to \zhalf{\rpbker{F}}\to 
0.
\]
  
Let $F$ be a field with discrete valuation and corresponding residue field $k$. There exists a 
surjective homomorphism, depending on the choice of a uniformizer $\pi$, 
\[
\hot{F}{\zhalf{\Z}}\to \hot{k}{\zhalf{\Z}}\oplus \zhalf{\rpbker{k}}.
\]
The main theorem of this article (see Theorem \ref{thm:mainsv} and Proposition 
\ref{prop:mainsv3}) describes conditions under which this map is an isomorphism and,
 more generally, when it has a kernel annihilated by $3$. Thus, for example, if the valuation 
is complete and $\chara{k}\not=2$, the kernel of this homomorphism will be annihilated by $3$.
More precise statements and examples can be found in sections \ref{sec:main} and \ref{sec:3tors} 
below. 

As an application of these results we calculate the third homology of $\mathrm{SL}_2$ of 
certain higher-dimensional local fields in terms of the indecomposable $K_3$ of the field 
and scissors congruence groups, in the original sense of Dupont and Sah, of the intermediate 
residue fields. These calculations, in particular, vastly generalize the calculation of 
$\ho{3}{\spl{2}{F}}{\zhalf{\Z}}$ for classical local fields found in \cite{hut:rbl11}. Unlike 
the results in \cite{hut:rbl11} our calculations here include the case of 
 $2$-adic local fields in 
characteristic zero.

A fundamental feature of the homology of $\mathrm{SL}_2$ of a field is that it is naturally 
acted on by the group of square classes, $\sq{F}$, of the multiplicative group is a field. 
Thus it is a module over the group ring $\sgr{F}:= \Z[\sq{F}]$. This module structure is integral 
to all the results and calculations below. Indeed, $\hot{F}{\zhalf{\Z}}=
\aug{F}\ho{3}{\spl{2}{F}}{\zhalf{\Z}}$, where $\aug{F}$ is the augmentation ideal 
in $\sgr{F}$ and thus all of our results concern the non-triviality of this group action. 
This non-triviality is used in \cite{hut:wendt} to demonstrate 
the failure of the property of \emph{weak homotopy invariance} 
for the third homology of $\mathrm{SL}_2$ of fields with non-trivial valuations. 

\begin{rem}
Let $F$ be a field with discrete valuation $v$ and residue field $k$. Let 
$\mathcal{O}_v = \{ x\in F\ |\ v(x)\geq 0\}$ 
be the associated valuation ring. Our main result below implies the existence, under certain 
conditions, of an exact sequence 
\[
0\to \hot{k}{\zhalf{\Z}}\to \hot{F}{\zhalf{\Z}}\to \zhalf{\rpbker{k}}\to 0.
\]

In a subsequent paper we will use this to show that under appropriate conditions there 
is a natural  short exact sequence
\[
0\to \ho{3}{\spl{2}{\mathcal{O}_v}}{\zhalf{\Z}} \to
\ho{3}{\spl{2}{F}}{\zhalf{\Z}}\to \zhalf{\rpbker{k}}\to 0. 
\]
(See \cite{hut:arxivhlr} for some cases of this result.) The main point is that the natural 
map $\mathcal{O}_v\to k$ induces an isomorphism 
$\hot{\mathcal{O}_v}{\zhalf{\Z}}\cong \hot{k}{\zhalf{\Z}}$ under appropriate circumstances. 

We observe that analogous results are known for the \emph{second} homology of 
$\mathrm{SL}_2$ in which the first Milnor-Witt $K$-group, $\mwk{1}{k}$, 
 of the residue field replaces
the refined scissors congruence group: 

Let $\mathcal{O}$ be a local domain with field of fractions $F$ and infinite residue field $k$. 
Then $\ho{2}{\spl{2}{F}}{\Z}\cong \mwk{2}{F}$ and 
$\ho{2}{\spl{2}{\mathcal{O}}}{\Z}\cong \mwk{2}{\mathcal{O}}$ (by \cite{sus:tors} for the case 
of an infinite field, and 
\cite[Theorem 5.37]{schlicht:arxiveuler} for local rings with infinite residue fields). 

On the other hand, the main theorem of \cite{gsz:mw} implies a short exact sequence 
\[
0\to \mwk{2}{\mathcal{O}}\to \mwk{2}{F}\to \mwk{1}{k}\to 0
\]
when $\mathcal{O}$ is regular and $1/2\in \mathcal{O}$. It follows that,
 under these conditions, there is a natural short exact sequence
\[
0\to \ho{2}{\spl{2}{\mathcal{O}}}{\Z}\to 
\ho{2}{\spl{2}{F}}{\Z} \to \mwk{1}{k}\to 0.
\]
\end{rem}

\begin{rem} 
For an infinite field $F$, $\ho{3}{\spl{3}{F}}{\zhalf{\Z}}\cong \zhalf{\kind{F}}$ 
and thus $\hot{F}{\zhalf{\Z}}$ can also be understood as the kernel of the
 stabilization homomorphism
\[
\ho{3}{\spl{2}{F}}{\zhalf{\Z}}\to \ho{3}{\spl{3}{F}}{\zhalf{\Z}} 
\]
(see Lemma \ref{lem:h3sl20}).

\end{rem}

\subsection{Layout of the article} In section \ref{sec:bloch} we review the definitions and 
relevant 
properties of (refined) scissors congruence groups and their relation to the third homology 
of $\mathrm{SL}_2$ of a field. 

In section \ref{sec:char} we prove an elementary character-theoretic local-global principle 
for modules over the group ring of an arbitrary 
elementary abelian $2$-group, such as the group of square classes of a field. This principle in 
turn is used to prove the main results of sections \ref{sec:rsc} and \ref{sec:main}.

Section \ref{sec:rsc} contains some further results on the refined scissors congruence group, 
$\rpbker{F}$. By definition this group is the kernel of a module homomorphism 
$\lambda_1:\rpb{F}\to \aug{F}^2$, where $\rpb{F}$ is an $\sgr{F}$-module described by generators 
and relations analogous to the description of the scissors congruence group $\pb{F}$. The main
result of this section is that $\zhalf{\rpbker{F}}$ is precisely the $+1$-eigenspace for the 
action of the square class of $-1\in F^\times$ on $\zhalf{\rpb{F}}$. It follows that 
$\zhalf{\rpbker{F}}$ is naturally a quotient of $\zhalf{\rpb{F}}$ and thus also 
admits a simple presentation as a $\sgr{F}$-module. 

In section \ref{sec:val} we review, from \cite{hut:rbl11}, 
 the specialization or residue homomorphism $\spec{v}$ on 
$\rpbker{F}$ associated to a general (Krull) valuation $v$ on a field $F$. We give a 
characterization of the kernel of the map $\spec{v}$ for general valuations and we show in 
the case of a discrete valuation with residue field $k$ it gives rise to a surjective 
homomorphism 
\begin{eqnarray}\label{eqn:intro}
\hot{F}{\zhalf{\Z}}\to \hot{k}{\zhalf{\Z}}\oplus \zhalf{\rpbker{k}}.
\end{eqnarray}     

In section \ref{sec:main} we prove the main theorem (Theorem \ref{thm:main}) of this 
article: Let $F$ be a field with discrete valuation $v$ and residue field $k$. Suppose that 
$F$ satisfies the additional condition that there exists $n\geq 1$ such that 
$U_n\subset (F^\times)^2$, where $U_n:=\{ x\in F^\times |\  v(x)=0,\  v(1-x)\geq n\}$. 
Then the residue homomorphism induces an isomorphism 
\[
\aug{F}\zhalf{\rrpbker{F}}\cong \aug{k}\zhalf{\rrpbker{k}}\oplus \zhalf{\rrpbker{k}} 
\]
where $\rrpbker{F}$ is a quotient of $\rpbker{F}$ by a certain submodule 
$\cconstmod{F}$ annihilated by $3$. 
This implies the the surjective homomorphism (\ref{eqn:intro}) has kernel annihilated by $3$.

The final section, section \ref{sec:3tors}, analyses further the $3$-torson module $\cconstmod{F}$ 
and determines general condition sunder which the kernel of the homomorphism 
(\ref{eqn:intro}) is trivial. We conclude with applications to the calculations of 
$\ho{3}{\spl{2}{F}}{\zhalf{\Z}}$ for certain classical and higher-dimensional local fields, 
such as iterated power series field $\laurs{\laurs{\Q_p}{x_{n-1}}\cdots}{x_1}$. 

In an appendix to the paper we revisit the specialization map $\spec{v}$ introduced in 
\cite{hut:rbl11} and show that it can be defined with an improved target. We use this 
improved version of the specialization homomorphism in section \ref{sec:3tors}.

\subsection{Some notation}
For a commutative unital ring $A$, $\unitr{A}$ denotes the group of units of $A$. 

For any abelian group $A$, we denote $A\otimes \znth{\Z}{n}$ by $\znth{A}{n}$. 

If $q$ is a prime power, $\F{q}$ will denote the finite field with $q$ elements.

For a group $\G$ and a $\Z[G]$-module $M$, $M_G$ will denote the module of coinvariants; 
$M_G=\ho{0}{G}{M}=M/\aug{G}{M}$, where $\aug{G}$ is the augmentation ideal of $\Z[G]$.
 
Given an abelian group $G$ we let  $\asym{2}{\Z}{G}$ denote the group 
\[
\frac{G\otimes_{\Z}G}{<x\otimes y + y\otimes x | x,y \in G>}
\]
and, for $x,y\in G$, we denote by $x\asymm y$ the image of $x\otimes y$ in $\asym{2}{\Z}{G}$.

For a field $F$, we let $\sgr{F}$ denote the group ring $\Z[\sq{F}]$ of the group 
of square classes of $F$ and we let $\aug{F}$ denote 
the augmentation ideal of $\sgr{F}$. If $x\in F^\times$, we denote the corresponding square-class, 
considered as an element of $\sgr{F}$, by $\an{x}$. The generators $\an{x}-1$ of $\aug{F}$ will be 
denoted $\pf{x}$. The particular element $\pp{-1}:=1+\an{-1}\in\sgr{F}$ will play a role below.

\section{Review: Scissors congruence Groups and Bloch Groups}\label{sec:bloch}
In this section we review the definitions and 
properties of (refined) scissors congruence groups and their relation to the third homology 
of $\mathrm{SL}_2$ of a field. 
\subsection{The Bloch group of a  field}
For a field $F$, with at least $4$ elements, the \emph{scissors congruence group} 
(also called the  \emph{pre-Bloch group}), $\pb{F}$, is the group 
generated 
by the elements $\gpb{x}$, $x\in F^\times\setminus\{ 1\}$,  subject to the relations 
\[
R_{x,y}:\quad  \gpb{x}-\gpb{y}+\gpb{\frac{y}{x}}-\gpb{\frac{1-x^{-1}}{1-y^{-1}}}+\gpb{\frac{1-x}{1-y}} \quad x\not=y.
\]

The map 
\[
\lambda:\pb{F}\to \asym{2}{\Z}{F^\times},\quad  [x]\mapsto \left(1-{x}\right)\asymm {x}
\]
is well-defined, and the \emph{Bloch group of $F$}, $\bl{F}\subset \pb{F}$, is 
defined to be the kernel of $\lambda$. 

The Bloch group is closely related to the indecomposable $K_3$ of the field $F$: 
For any field $F$ there is a natural exact sequence 
\[
0\to \widetilde{\Tor{\mu_F}{\mu_F}}\to \kind{F}\to \bl{F} \to 0
\]
where $ \widetilde{\Tor{\mu_F}{\mu_F}}$ is the unique nontrivial extension of $\Tor{\mu_F}{\mu_F}$ 
by $\Z/2$. (See 
Suslin \cite{sus:bloch} for infinite fields and \cite{hut:cplx13} for finite fields.)

\subsection{The refined Bloch group and $\ho{3}{\spl{2}{F}}{\Z}$}\label{subsec:rbl}
For any field $F$ there is a natural surjective homomorphism
\begin{eqnarray}\label{eqn:kind}
\hoz{3}{\spl{2}{F}}\to \kind{F}.
\end{eqnarray}

When $F$ is quadratically closed (i.e. when $\sq{F}=1$) this map is an isomorphism. 
However, in general, the group extension 
\[
1\to \spl{2}{F}\to\gl{2}{F}\to F^\times\to 1
\]
induces an action -- by conjugation -- of $F^\times$ on $\hoz{\bullet}{\spl{2}{F}}$ which factors 
through $\sq{F}$. It 
can be shown that the map (\ref{eqn:kind}) induces an isomorphism
\[
\ho{3}{\spl{2}{F}}{\zhalf{\Z}}_{\sq{F}}\cong\zhalf{\kind{F}}
\]
(see \cite{mirzaii:third}), but -- as our calculations in \cite{hut:rbl11} and below show -- the action 
of $\sq{F}$ on 
$\hoz{3}{\spl{2}{F}}$ is in general non-trivial.

Thus
$\hoz{3}{\spl{2}{F}}$ is naturally an $\sgr{F}$-module, and for general fields, 
in order to give a Bloch-type description
of it, we must incorporate the $\sgr{F}$-module structure at each stage of the process.

For a field $F$ with at least $4$ elements, $\rpb{F}$ is defined to be the 
  $\sgr{F}$-module with generators $\gpb{x}$, $x\in F^\times$ 
subject to the relations $\gpb{1}=0$ and 
\[
S_{x,y}:\quad 0=\gpb{x}-\gpb{y}+\an{x}\gpb{ \frac{y}{x}}-\an{x^{-1}-1}\gpb{\frac{1-x^{-1}}{1-y^{-1}}}
+\an{1-x}\gpb{\frac{1-x}{1-y}},\quad x,y\not= 1
\]

Of course, from the definition it follows immediately that 
$\pb{F}=(\rpb{F})_{\sq{F}}=\ho{0}{\sq{F}}{\rpb{F}}$.

Let  $\bw= (\lambda_1,\lambda_2)$ 
be the $\sgr{F}$-module homomorphism 
\[
\rpb{F}\to \aug{F}^2\oplus\asym{2}{\Z}{F^\times}
\]
where 
$\lambda_1:\rpb{F}\to \aug{F}^2$ is the map $\gpb{x}\mapsto \pf{1-x}\pf{x}$, and $\lambda_2$ is the composite
\[
\xymatrix{
\rpb{F}\ar@{>>}[r]
&\pb{F}\ar[r]^-{\lambda}
&\asym{2}{\Z}{F^\times}.
}
\] 

It can be shown that $\bw$ is well-defined. 

The  
\emph{refined Bloch group} of the field $F$ (with at least $4$ elements) to be the $\sgr{F}$-module 
\[
\rbl{F}:=\ker{\bw: \rpb{F}\to \aug{F}^2\oplus\asym{2}{\Z}{F^\times}}.
\]

For the fields with $2$ and $3$ elements  the following \textit{ad hoc} definitions allow us to include 
these fields in the statements of some of our results. 

$\pb{\F{2}}=\rpb{\F{2}}=\rbl{\F{2}}=\bl{\F{2}}$ 
is simply an additive group of order $3$ with distinguished generator, denoted $\bconst{\F{2}}$.

$\rpb{\F{3}}$ is the cyclic $\sgr{\F{3}}$-module generated by the symbol $\gpb{-1}$ and subject to the one relation
\[
0=2\cdot(\gpb{-1}+\an{-1}\gpb{-1}).
\] 
The homomorphism
\[
\bw:\rpb{\F{3}}\to\aug{\F{3}}^2=2\cdot \Z\pf{-1}
\]
is the $\sgr{\F{3}}$-homomorphism sending $\gpb{-1}$ to $\pf{-1}^2=-2\pf{-1}$. 

Then $\rbl{\F{3}}=\ker{\bw}$ is the submodule of order $2$ generated by $\gpb{-1}+\an{-1}\gpb{-1}$.

Furthermore, we let $\pb{\F{3}}=\rpb{\F{3}}_{\F{3}^\times}$. 
This is a cyclic $\Z$-module of order $4$ with generator 
$\gpb{-1}$. Let $\lambda:\pb{\F{3}}\to \asym{2}{\Z}{\F{3}^\times}$ 
be the map $\gpb{-1}\mapsto -1\asymm -1$. Then 
$\bl{\F{3}}:=\ker{\lambda}=\rbl{\F{3}}$.  


We recall some results from \cite{hut:cplx13}: The main result there is 

\begin{thm}\label{thm:main} Let $F$ be any field.

There is a natural complex 
\[
0\to \Tor{\mu_F}{\mu_F}\to\ho{3}{\spl{2}{F}}{{\Z}}\to{\rbl{F}}\to 0.
\]
which is exact everywhere except possibly at the middle term. The middle homology is annihilated by $4$.

In particular, for any field there is a natural short exact sequence
\[
0\to\zhalf{\Tor{\mu_F}{\mu_F}}\to\ho{3}{\spl{2}{F}}{\zhalf{\Z}}\to\zhalf{\rbl{F}}\to 0.
\]
\end{thm}

The following result is Corollary 5.1 in \cite{hut:cplx13}:

\begin{lem}
\label{lem:blochinf}
Let $F$ be a field. Then the natural map $\rbl{F}\to\bl{F}$ is surjective and 
the induced map 
$\rbl{F}_{F^\times}\to \bl{F}$ has a $2$-primary torsion kernel.
\end{lem}

Now for any field $F$, let 
\begin{eqnarray*}
\hot{F}{\zhalf{\Z}}:=\ker{\ho{3}{\spl{2}{F}}{\Z}\to\kind{F}}
\end{eqnarray*}
and
\begin{eqnarray*}
\rblker{F}:=\ker{\rbl{F}\to\bl{F}}
\end{eqnarray*}

The following is Lemma 5.2 in \cite{hut:cplx13}. \begin{lem} \label{lem:h3sl20}
Let $F$ be a field. Then
\begin{enumerate}
\item $\hot{F}{\zhalf{\Z}}\cong\zhalf{\rblker{F}}$ as $\zhalf{\sgr{F}}$-modules.
\item $\hot{F}{\zhalf{\Z}}=\aug{F}\ho{3}{\spl{2}{F}}{\zhalf{\Z}}$ and
$\zhalf{\rblker{F}}=\aug{F}\zhalf{\rbl{F}}$.
\item 
\begin{eqnarray*}
 \hot{F}{\zhalf{\Z}}&=& \ker{\ho{3}{\spl{2}{F}}{\zhalf{\Z}}\to \ho{3}{\spl{3}{F}}{\zhalf{\Z}}}\\
&=& \ker{\ho{3}{\spl{2}{F}}{\zhalf{\Z}}\to \ho{3}{\gl{2}{F}}{\zhalf{\Z}}}
\end{eqnarray*}
\end{enumerate}
\end{lem}

\begin{lem}\label{lem:rbl0}
If the field $F$ is quadratically closed, real-closed or finite then $\aug{F}\rbl{F}=0$.
 
Hence 
\[
\hot{F}{\zhalf{\Z}}=0\quad \mbox{  and }\quad 
\ho{3}{\spl{2}{F}}{\zhalf{\Z}}\cong\zhalf{\kind{F}}.
\]
\end{lem}
\begin{proof}
Since $\hoz{3}{\spl{2}{F}}$ maps onto $\rbl{F}$ as an $\sgr{F}$-module, this follows from the fact that $\sq{F}$ 
acts trivially on  $\hoz{3}{\spl{2}{F}}$ in each of these cases. 

For a quadratically closed field this is vacuously true.

This result is proved by Parry and Sah in \cite{sah:parry} for the field $\R$, but their argument extends easily to 
any real-closed field.

For finite fields, the relevant result is Lemma 3.8 of \cite{hut:cplx13}.
\end{proof}

\begin{rem}\label{rem:fin}
In fact it is shown in the final section of \cite{hut:cplx13} when $F$ is the finite field 
$\F{q}$ with $q$ elements that $\rbl{\F{q}}=\bl{\F{q}}$ and that 
$ \bl{\F{q}}$ is cyclic of order $q+1$ or $(q+1)/2$ according as $q$ is even or odd.  

Note further that $\card{\pb{\F{q}}/\bl{\F{q}}} = 1$ or $2$ and thus 
\[
\zhalf{\bl{\F{q}}}\cong\zhalf{\pb{\F{q}}}\cong \zhalf{\left(\Z/(q+1)\right)}.
\]
\end{rem}
\subsection{The elements $\suss{i}{x}$}

In \cite{sus:bloch} Suslin defines the elements $\sus{x}:=\sgpb{x}+\sgpb{x^{-1}}\in \pb{F}$ and 
shows that they 
satisfy 
\[
\sus{xy}=\sus{x}+\sus{y} \mbox{ and }     2\sus{x}=0  \forall  x,y\in F^\times.
\]
In particular, $\sus{x}=0$ in $\zhalf{\pb{F}}$.

There are two natural liftings of these elements to $\rpb{F}$: given $x\in F^\times$ we define 
\[
\suss{1}{x}:=\sgpb{x}+\an{-1}{\sgpb{x^{-1}}}
\]
and 
\[
\suss{2}{x}:=\left\{
\begin{array}{ll}
\an{1-x}\left(\an{x}\sgpb{x}+\sgpb{x^{-1}}\right),& x\not= 1\\
0,& x=1
\end{array}
\right.
\]

(If $F=\F{2}$, we interpret this as $\suss{i}{1}=0$ for $i=1,2$. For $F=\F{3}$, we have $\suss{1}{-1}=\suss{2}{-1}
=\gpb{-1}+\an{-1}\gpb{-1}$. )
 
We summarize here some of the basic facts about these elements of $\rpb{F}$ (see \cite{hut:rbl11})
\begin{lem} \label{lem:sus1}
For $i=1,2$ we have 
\begin{enumerate}
\item $\suss{i}{xy}=\an{x}\suss{i}{y}+\suss{i}{x}$ for all $x,y\in F^\times$.
\item $\lambda_1(\suss{i}{x})=-\pp{-1}\pf{x}=\pf{-x}\cdot\pf{x}\in \aug{F}^2$ for all $x\in F^\times$.
\item $\lambda_2(\suss{i}{x})=\lambda(\sus{x})=x\asymm (-x)\in \asym{2}{\Z}{F^\times}$ for all $x\in F^\times$.
\end{enumerate}
\end{lem}

\subsection{The refined scissors congruence group}
In this article, the main object of study, however, is the functor
\[
\rpbker{F}:=\ker{\lambda_1:\rpb{F}\to \aug{F}^2}
\]
which we call the \emph{refined scissors congruence group} of $F$. 

The refined scissors congruence group
 lies between $\rbl{F}$ and $\rpb{F}$. From the definition, $\rbl{F}$ is the 
kernel of the $\sgr{F}$-homomorphism
\[
\lambda_2:\rpbker{F}\to \asym{2}{\Z}{F^\times}.
\]

Furthermore, from Lemma \ref{lem:blochinf} above, we deduce 
\begin{lem}\label{lem:qblochinf}
For any field $F$, the cokernel of the natural map $\rpbker{F}\to \pb{F}$ is annihilated by 
$2$  and the induced map 
$\rpbker{F}_{F^\times}\to \pb{F}$ has $2$-primary torsion kernel. 

Furthermore, there is a natural 
short exact sequence of $\zhalf{\sgr{F}}$-modules
\[
0\to \aug{F}\zhalf{\rpbker{F}}\to \zhalf{\rpbker{F}}\to \zhalf{\pb{F}}\to 0
\]
and a natural isomorphism 
\[
\aug{F}\zhalf{\rbl{F}}\cong \aug{F}\zhalf{\rpbker{F}}.
\]
\end{lem}

\begin{proof}
Let $\symb{F}$ denote the image of the map $\lambda:\pb{F}\to \asym{2}{\Z}{F^\times}$. 
By Lemma \ref{lem:sus1} (2), the element 
\[
f(x):= \pp{-1}(\gpb{x})+\pf{1-x}\suss{1}{x}
\]
lies in $\rpbker{F}$ and, since the action of $F^\times$ on $\asym{2}{\Z}{F^\times}$ is 
trivial, we have 
\[
\lambda_2(f(x))=\lambda_2(\pp{-1}\gpb{x})=2\lambda(\gpb{x}).
\]

There 
is a natural commutative diagram with exact rows
\[
\xymatrix{
0\ar[r]
&
\rbl{F}\ar[r]\ar@{>>}[d]
&
\rpbker{F}
\ar^-{\lambda_2}[r]\ar[d]
&\symb{F}\ar^-{=}[d]
&
\\
0\ar[r]&
\bl{F}
\ar[r]
&
\pb{F}
\ar^-{\lambda}[r]
&
\symb{F}\ar[r]
&
0.
}
\]
from which the first statement of the Lemma follows.

After tensoring the terms in this diagram with $\zhalf{\Z}$, 
the remaining statements follow from Lemma 
\ref{lem:blochinf}. 
\end{proof}

Combining this with Lemma \ref{lem:h3sl20}, we immediately deduce
\begin{cor}\label{cor:rpbker}
For any field $F$ there is a natural isomorphism
\[
\hot{F}{\zhalf{\Z}}\cong \aug{F}\zhalf{\rpbker{F}}.
\]
\end{cor}
\subsection{The module $\cconstmod{F}$}\label{sec:cconst}

In \cite{sus:bloch}, Suslin shows that the elements $\gpb{x}+\gpb{1-x}\in \bl{F}\subset \pb{F}$ are 
independent 
of $x$ and that the resulting constant, $\pbconst{F}$ has order dividing $6$. 
Furthermore Suslin shows 
that $\pbconst{\R}$ has exact order $6$.  

In \cite{hut:rbl11} it is shown that the elements 
\[
C(x)=\gpb{x}+\an{-1}\gpb{1-x}+\pf{1-x}\suss{1}{x}\in\rbl{F}
\]
are constant (for a field with at least $4$ elements) and have order dividing $6$. 

When $F$ has at least $4$ elements, we 
 denote this constant  $\bconst{F}$. Of course, $\bconst{F}$ 
maps to $\pbconst{F}$ under the natural map $\rpb{F}\to\pb{F}$. 

Furthermore, we let $\bconst{\F{2}}$ denote the distinguished generator 
of $\rbl{\F{2}}$ of order $3$ and we set $\bconst{\F{3}}:= \suss{1}{-1}=(1+\an{-1})\gpb{-1}$.

We let $\cconst{F}:=2\bconst{F}$ for any field $F$. Thus either $\cconst{F}=0$ or it has order $3$.
 
\begin{prop} For a field $F$ we have $\cconst{F}=0$ if and only if the polynomial $X^2-X+1=0$ has
a root in $F$. 
\end{prop}
\begin{proof}
If $x\in F$ satisfies $x^2-x+1=0$, then $1-x=x^{-1}$ and the argument of \cite{hut:rbl11} Corollary 3.9 
shows that $\bconst{F}=\suss{1}{-1}$ which has order dividing $2$.

Now if $\mathrm{char}(F)=3$ then $-1$ is a root of $X^2-X+1$ and hence $\cconst{F}=0$. 

So we may assume that $\mathrm{char}(F)\not=3$. Thus $X^2-X+1$ has a root in $F$ if and only if $F$ contains 
a primitive cube root of unity, $\zeta_3$.   Let $k$ be the prime subfield of $F$. If $k=\Q$, then $2\pbconst{k}=
2\pbconst{\Q}\in \bl{\Q}$ has order $3$ (since $\Q$ embeds in $\R$). On the other hand, 
if $k=\F{p}$ then $2\pbconst{k}$ has order $3$ if and only if $p\equiv -1\pmod{3}$ by 
\cite{hut:cplx13} Lemma 7.11. 
Thus, in all cases, $2\pbconst{k}=0$ in $\bl{k}$ if and only if $\zeta_3\in k$. 

Of course, $\pbconst{k}\mapsto \pbconst{F}$ under the map $\bl{k}\to\bl{F}$. But, by Suslin's Theorem 
(\cite{sus:bloch} Theorem 5.2)  we have a commutative diagram with exact rows
\[
\xymatrix{
0\ar[r]
&
\widetilde{\Tor{\mu_k}{\mu_k}}\ar[r]\ar[d]
&
\kind{k}\ar[r]\ar@{^{(}->}[d]
&\bl{k}\ar[r]\ar[d]
&
0\\
0\ar[r]&
\widetilde{\Tor{\mu_F}{\mu_F}}\ar[r]
&
\kind{F}\ar[r]
&
\bl{F}\ar[r]
&
0
}
\] 
where the middle vertical arrow is injective 
(\cite{levine:k3ind}, Corollary 4.6). If $2\pbconst{k}\not=0$, then 
it lies in the kernel of $\bl{k}\to\bl{F}$ only if $3|\mu_F$. Thus $2\pbconst{F}=0$ only if $\zeta_3\in F$. 
\end{proof}

We recall also the following (\cite{hut:rbl11} Theorem 3.13):

\begin{thm} \label{thm:df}
Let $F$ be a field. Then 
\begin{enumerate}
\item For all $x\in F^\times$, $\pf{x}\cconst{F}=\suss{1}{x}-\suss{2}{x}$.
\item Let $E$ the field obtained from $F$ by adjoining a root of $X^2-X+1$. 

 Then $\pf{x}\cconst{F}=0$ if $x\in \pm N_{E/F}(E^\times)\subset F^\times$. 
\end{enumerate}
\end{thm}  

We let $\cconstmod{F}$ denote the cyclic 
$\sgr{F}$-submodule of $\rbl{F}$ generated by $\cconst{F}$. 
Of course, 
since $3\cconst{F}=0$ always, in fact $\cconstmod{F}$ is a module over the 
group ring $\F{3}[\sq{F}]$. 

Furthermore, by Theorem \ref{thm:df} (2), we have $\an{x}\cconst{F}=\cconst{F}$ if 
$x\in \tnorm{F}:=\pm N_{E/F}(E^\times)$. 

Let $\rsq{F}=F^\times/\tnorm{F}$. Thus the action of 
$\sq{F}$ on $\cconst{F}$ factors through the quotient $\rsq{F}$, and hence $\cconstmod{F}$ is a 
cyclic module over the ring $\rsgr{F}:=\F{3}[\rsq{F}]$.

\begin{rem} 
We will show in section \ref{sec:3tors} below that if $F$ is a local field with finite 
residue field
for which $\cconst{F}\not=0$  -- and supposing that $[F:\Q_3]$ is odd in the case that  
$\Q_3\subset F$ -- 
then $\cconstmod{F}$ is a free rank one $\rsgr{F}$-module.
In this case it follows that we have a 
converse to (2) of Theorem \ref{thm:df}:
$\pf{x}\cconst{F}=0$ if and only $x\in\tnorm{F}$. 
(And hence, over such fields, $\suss{1}{x}=\suss{2}{x}$  
if and only if $x\in\tnorm{F}$.) 

However it seems unlikely that the corresponding statement is true for global fields 
such as $F=\Q$. 
Indeed, an affirmative answer to the question at the end of the introductory 
section of \cite{hut:rbl11}
would imply that the kernel of the surjective map $\rsgr{\Q}\to \cconstmod{\Q}$ is infinite.   
\end{rem}

\subsection{Reduced Scissors Congruence Groups}
For a field $F$, let $K(F)$ denote the subgroup of $\pb{F}$ generated by 
the elements $\sus{x}$ of order 
$2$. Let $\qpb{F}$ denote the quotient group $\pb{F}/K(F)$. 

Observe that 
$\lambda(\sus{x})=(-x)\asymm x$ for all $x\not= 0,1$. Thus is we set 
\[
\qasym{2}{\Z}{F^\times}
:=\asym{2}{\Z}{F^\times}/\an{ x\asymm (-x)| x\in F^\times}
\]
then there is a well-defined homomorphism 
\[
\tilde{\lambda}:\qpb{F}\to \qasym{2}{\Z}{F^\times},\quad \gpb{x}\mapsto (1-x)\asymm x.
\]
We let $\qbl{F}:= \ker{\tilde{\lambda}:\qpb{F}\to \qasym{2}{\Z}{F^\times}}$. 

Of course, since $K(F)$ is annihilated by $2$, the natural homomorphisms $\pb{F}\to\qpb{F}$ and 
$\bl{F}\to\qbl{F}$ induce isomorphisms when $2$ is inverted:
\[
\zhalf{\pb{F}}=\zhalf{\qpb{F}}\mbox{ and } \zhalf{\bl{F}}=\zhalf{\qbl{F}}.
\] 

For $i=1,2$, let $\ks{i}{F}$ denote the $\sgr{F}$-submodule of $\rpb{F}$ generated by the set 
$\{ \suss{i}{x}\ |\ x\in F^\times\}$.

\begin{lem}\cite[Lemma 3.3]{hut:rbl11} \label{lem:kf}
Let $F$ be a field.
 Then for $i\in \{ 1,2\}$
\[
\lambda_1\left( \ks{i}{F}\right) = \pp{-1}(\aug{F})\subset \aug{F}^2
\]
and $\ker{\lambda_1\res{\ks{i}{F}}}$ is annihilated by $4$.
\end{lem}

Let 
\[
\qrpb{F}:=\rpb{F}/\ks{1}{F} 
\] 
Then the  $\sgr{F}$-homomorphism
\[
\rbw=(\tilde{\lambda}_1,\tilde{\lambda}_2):
\qrpb{F}\to \frac{\aug{F}^2}{\pp{-1}\aug{F}}\oplus \qasym{2}{\Z}{F^\times}, 
\gpb{x}\mapsto (\pf{x}\pf{1-x},x\asymm (1-x))
\]
is well-defined. We set $\qrbl{F}:=\ker{\rbw}$.

We let
\[
\qrpbker{F}:= \ker{\tilde{\lambda}_1:\qrpb{F}\to \frac{\aug{F}^2}{\pp{-1}\aug{F}}}.
\]

From Lemma \ref{lem:kf} we immediately deduce

\begin{cor}\label{cor:qrbl}
The natural maps $\rbl{F}\to\qrbl{F}$ and $\rpbker{F}\to \qrpbker{F}$ are surjective with 
kernel annihilated by $4$. In particular,
\[
\zhalf{\rbl{F}}=\zhalf{\qrbl{F}} \mbox{ and } \zhalf{\rpbker{F}}=\zhalf{\qrpbker{F}}.
\]
\end{cor}

Since 
$\cconstmod{F}$ is annihilated by $3$, it follows that the composite 
\[
\cconstmod{F}\to\rbl{F}\to\qrbl{F}
\]  
is injective and we will identify $\cconstmod{F}$ with its image in $\qrbl{F}$.


We further define
\[
\rrpb{F}:=\frac{\qrpb{F}}{\cconstmod{F}},\quad \rrpbker{F}:=\frac{\qrpbker{F}}{\cconstmod{F}}
\mbox{ and }\quad \rrbl{F}:=\frac{\qrbl{F}}{\cconstmod{F}}.
\]


Since $\cconstmod{F}\subset \rbl{F}=\ker{\bw}$, it follows that $\lambda_1$ induces a well-defined 
homomorphism 
\[
\bar{\lambda}_1:\rrpb{F}\to \frac{\aug{F}^2}{\pp{-1}\aug{F}}
\]
and $\rrpbker{F}= \ker{\bar{\lambda}_1}$. 

Analogously, we also define $\rredpb{F}:= \qpb{F}/\Z\cdot\cconst{F}$. 

Thus we have natural short exact sequences 
\begin{eqnarray*}
&0\to \cconstmod{F}\to \qrbl{F}\to \rrbl{F}\to 0&\\
&0\to \cconstmod{F}\to \qrpbker{F}\to \rrpbker{F}\to 0&\\
&0\to \Z\cdot \cconst{F}\to \qpb{F}\to \rredpb{F}\to 0&\\
\end{eqnarray*}

Since $3\cdot \cconst{F}=0 $ and $\cconst{F}=0$ in both $\rpb{F}$ and $\pb{F}$ 
when $X^2-X+1$ has a root in 
$F$, we have:

\begin{lem}\label{lem:rpbker3}\ 

\begin{enumerate}
\item For any field $F$
\[
\znth{\qrpbker{F}}{3}= \znth{\rrpbker{F}}{3} \mbox{ and } 
\znth{\qpb{F}}{3}= \znth{\rredpb{F}}{3}
\]
\item If $X^2-X+1$ has a root in $F$ then
\[
\qrpbker{F}=\rrpbker{F}\mbox{ and } \qpb{F}=\rredpb{F}.
\]
\end{enumerate}
\end{lem}

\begin{cor}\label{cor:rpbker3}\ 

\begin{enumerate}
\item For any field $F$
\[
\hot{F}{\znth{\Z}{6}}\cong\aug{F}\znth{\rrpbker{F}}{6}.
\]
\item If $X^2-X+1$ has a root in $F$ then
\[
\hot{F}{\znth{\Z}{2}}\cong\aug{F}\znth{\rrpbker{F}}{2}.
\]
\end{enumerate}

\end{cor}

Observe also that $\ks{2}{F}\subset\ks{1}{F}+\icconstmod{F}$ by Theorem \ref{thm:df} (2). 
It follows that for 
$i=1,2$ and all $x$ we have 
$\suss{i}{x}=0$ in $\rrpb{F}$. 

\begin{lem}\label{lem:pb-} Let $F$ be a field. In $\rrpb{F}$ we have 
\begin{enumerate}
\item For all $x$, $\sgpb{x^{-1}}=-\an{-1}\gpb{x}=-\an{x}\gpb{x}$.
\item For all $x$, $\pf{y}\gpb{x}=0$  whenever $y\equiv -x\pmod{(F^\times)^2}$.
\item For all $x$, $\gpb{1-x}=\an{-1}\gpb{x}=\sgpb{x^{-1}}$.
\end{enumerate}
\end{lem}
\begin{proof}
\begin{enumerate}
\item The first equality follows from $\suss{1}{x}=0$, the second from $\suss{2}{x}=0$.
\item From $\an{x}\gpb{x}=\an{-1}\gpb{x}$ it follows that $\an{y}\gpb{x}=\an{-x}\gpb{x}=\gpb{x}$.
\item Since $\bconst{F}\equiv-\cconst{F}\pmod{\ks{1}{F}}$ (\cite{hut:rbl11}, Lemma 4.1) 
it follows that $\bconst{F}=0$ in $\rrpb{F}$, and thus we have (from the definition of 
$\bconst{F}$) that $0=\gpb{x}+\an{-1}\gpb{1-x}$ in $\rrpb{F}$.
\end{enumerate}
\end{proof}

\section{Characters and  $\zhalf{\sgr{F}}$-modules.}\label{sec:char}

Let $G$ be a multiplicative abelian group such that $g^2=1$ for all $g\in G$. 
We will let $\mathcal{R}$ denote the integral group ring $\Z[G]$.

For a character $\chi\in \dual{G}:=\hmz{G}{\mu_2}$, 
let $\upp{\mathcal{R}}{\chi}$ be the ideal of 
$\mathcal{R}$ generated
by the elements $\{ g-\chi(g)\ |\ g\in G\}$. In other words $\upp{\mathcal{R}}{\chi}$ is the 
kernel of the ring 
homomorphism $\rho(\chi):\mathcal{R}\to \Z$ sending $g$ to $\chi(g)$ for any $g\in G$. 
We let $\dwn{\mathcal{R}}{\chi}$ denote the associated $\mathcal{R}$-algebra structure on $\Z$; ie. 
$\dwn{\mathcal{R}}{\chi}:= \mathcal{R}/\upp{\mathcal{R}}{\chi}$. 

If $M$ is an $\mathcal{R}$-module, we let $\upp{M}{\chi}=
\upp{\mathcal{R}}{\chi}M $ and we let 
\[
\dwn{M}{\chi}:=M/\upp{M}{\chi}=
\left(\mathcal{R}/\upp{\mathcal{R}}{\chi}\right)\otimes_{\mathcal{R}} M
=\dwn{(\mathcal{R})}{\chi}\otimes_{\mathcal{R}}  M.
\]

Thus $\dwn{M}{\chi}$ is the largest quotient module of $M$ with the property that 
$g\cdot m =\chi(g)\cdot m$ for all $g\in G$.

In particular, if $\chi=\chi_0$, the trivial character,
 then $\upp{\mathcal{R}}{\chi_0}$ is the augmentation ideal $\aug{G}$, 
$\upp{M}{\chi_0}=\aug{G}M$  and $\dwn{M}{\chi_0}=M_{G}$.

More generally, for any finite subset $H\subset \dual{G}$ and for any $\mathcal{R}$-module 
$M$ we define
\[
\upp{M}{H}:= \bigcap_{\chi\in H}\upp{M}{\chi},\quad \dwn{M}{H}:= M/\upp{M}{H}. 
\]
\begin{prop}\label{prop:upph}
 Let $M$ be a $\zhalf{\mathcal{R}}$-module and let $H$ be a finite subset of $\dual{G}$. 

Then $\upp{M}{H}= \left(\upp{\mathcal{R}}{H}\right)M$ 
and there is a natural isomorphism of $\mathcal{R}$-modules of 
\[
\dwn{M}{H}\cong \oplus_{\chi\in H}\left(\dwn{M}{\chi}\right). 
\]
\end{prop}

\begin{proof}
We begin by observing that if $\chi_1\not=\chi_2$ in $\dual{G}$ then there exists $g\in G$ with\\  
$\chi_1(g)=-\chi_2(g)$. Thus it follows that 
\[
2= (1-\chi_1(g)g)+(1-\chi_2(g)g)\in \upp{\mathcal{R}}{\chi_1}+\upp{\mathcal{R}}{\chi_2}
\] 
and therefore  
$\upp{\zhalf{\mathcal{R}}}{\chi_1}+\upp{\zhalf{\mathcal{R}}}{\chi_2}= \zhalf{\mathcal{R}}$.

By the Chinese Remainder Theorem, it follows that 
\[
\upp{\mathcal{R}}{H}:=\cap_{\chi\in H} \upp{\mathcal{R}}{\chi}= 
\prod_{\chi\in H}\upp{\mathcal{R}}{\chi}
\]
 and that the maps
 $\rho(\chi):\zhalf{\mathcal{R}}\to \dwn{\zhalf{\mathcal{R}}}{\chi}$
 induce a short exact sequence of $\zhalf{\mathcal{R}}$-modules 
\[
0\to \upp{\zhalf{\mathcal{R}}}{H}\to \zhalf{\mathcal{R}}
\to \prod_{\chi\in H}\dwn{\zhalf{\mathcal{R}}}{\chi}\to 0. 
\]

Tensoring with $M$ now gives the  exact sequence
\[
\upp{R}{H}\otimes_R M\to M\to \oplus_{\chi\in H}\left(\dwn{M}{\chi}\right)\to 0.
\]
Since the image of the left-hand map is $\left(\upp{R}{H}\right)M$ and the kernel of 
the second map is\\
 $\bigcap_{\chi\in H}\upp{M}{\chi}=\upp{M}{H}$, 
the statements of the lemma follow. 
\end{proof}

\begin{cor}\label{cor:upph}
 Let $M$ be a $\zhalf{\mathcal{R}}$-module and let $H$ be a finite subset of $\dual{G}$. 
For any $\chi\in \dual{G}$ we have 
\[
\dwn{\left(\upp{M}{H}\right)}{\chi}=
\left\{
\begin{array}{ll}
\dwn{M}{\chi},& \chi\not\in H\\
\{ 0\},& \chi\in H. 
\end{array}
\right.
\]
\end{cor}

\begin{proof}
If $\chi\in H$, then $\upp{(\upp{M}{H})}{\chi}=\upp{M}{H}$ by definition, and the statement follows.

If $\chi\not\in H$, then the statement
 follows by applying the snake lemma to the commutative diagram with 
exact rows
\[
\xymatrix{
0\ar[r] 
&\upp{M}{H\cup \{ \chi\}}\ar[r]\ar[d]
&M\ar[r]\ar^-{=}[d]
&\oplus_{\psi\in H\cup \{ \chi\}}\dwn{M}{\psi}\ar[r]\ar[d]
&0\\
0\ar[r] 
&\upp{M}{H}\ar[r]
&M\ar[r]
&\oplus_{\psi\in H}\dwn{M}{\psi}\ar[r]
&0\\
}
\]
\end{proof}

Given $g\in G$ and $\chi\in \dual{G}$, 
let $\idm{g}{\chi}\in\zhalf{\mathcal{R}}$ denote 
the idempotent 
\[
\frac{1+\chi(g)\an{g}}{2}
\]
so that $g\cdot\idm{g}{\chi}=\chi(g)\idm{g}{\chi}$ for all $g$ and $\chi$. 

More generally, for any finite subset $S$ of $G$ and any $\chi\in\dual{G}$, let
\[
\idm{S}{\chi}:=\prod_{g\in S}\idm{g}{\chi}
=\prod_{g\in S}\left(\frac{1+\chi(g)\an{g}}{2}\right)
\in\zhalf{\mathcal{R}}.
\]

Note that, in practice, we may assume, when required, that $S$ is a \emph{subgroup} of $G$:

\begin{lem} Let $S\subset G$ be a finite set and let $\chi\in \dual{G}$. Let $T$ be the subgroup 
of $G$ generated by $S$. Then $\idm{S}{\chi}=\idm{T}{\chi}$.
\end{lem}
\begin{proof} Let $g\in T$. Then there exist $g_1,\ldots,g_r\in S$ with $g=\prod_ig_i$. 

Now $g_i\cdot\idm{S}{\chi}=\chi(g_i)\idm{S}{\chi}$ for each $i$. It follows that $\an{g}\idm{S}{\chi}
=\chi(g)\idm{S}{\chi}$ and hence that $\idm{S}{\chi}=\idm{g}{\chi}\idm{S}{\chi}$.
\end{proof}

When $G$ itself is finite, we let $\idem{\chi}:=\idm{G}{\chi}$. 

\begin{lem}
Suppose that $|G|<\infty$ and $M$ is an $\zhalf{\mathcal{R}}$-module. For any character $\chi$ the 
composite
\[
\xymatrix{
\idem{\chi}M\ar@{^(->}[r]
&M\ar@{>>}[r]
&\dwn{M}{\chi}}
\] 
is an isomorphism of $\zhalf{\mathcal{R}}$-modules. 
\end{lem}
\begin{proof}
Since $1-\idem{\chi}\in \upp{\left(\zhalf{\mathcal{R}}\right)}{\chi}$ by \ref{lem:upp}, 
while $\idem{\chi}\cdot \upp{\left(\zhalf{\mathcal{R}}\right)}{\chi}=0$,
 it follows that $\upp{\left(\zhalf{\mathcal{R}}\right)}{\chi}=
(1-\idem{\chi})\zhalf{\mathcal{R}}$ and more generally that 
 $\upp{M}{\chi}=(1-\idem{\chi})M$. 

Thus $M = \idem{\chi}M\oplus (1-\idem{\chi})M = 
\idem{\chi}M\oplus\upp{M}{\chi}$ as an $\zhalf{\mathcal{R}}$-module. 
\end{proof}

\begin{lem}\label{lem:upp}
 For any character $\chi$ and finite $S\subset G$, 
$1-\idm{S}{\chi}\in \upp{\left(\zhalf{\mathcal{R}}\right)}{\chi}$.
\end{lem}

\begin{proof} 
This follows from $\rho(\chi)(\idm{a}{\chi})=1$ for any $g\in G$. Thus 
$\rho(\chi)(\idm{S}{\chi})=1$. 
\end{proof}

\begin{cor} \label{cor:upp}
For any $\chi\in\dual{G}$ and $\zhalf{\mathcal{R}}$-module $M$, we have 
\[
\upp{M}{\chi}=\{ m\in M\ |\ \idm{S}{\chi}m=0 \mbox{ for some finite } S\subset G\}.
\]
\end{cor}

\begin{proof}
Let $m\in \upp{M}{\chi}$. $\upp{\mathcal{R}}{\chi}$ is the ideal of $\zhalf{\mathcal{R}}$ generated by the elements 
\[
1-\idm{g}{\chi}= \frac{1-\chi(g)\an{g}}{2}.
\]
 Thus there exist $g_1,\ldots g_r\in G$ such that 
\[
m=\sum_{i=1}^r(1-\idm{g_i}{\chi})m_i
\]
from which it follows that $\idm{S}{\chi}m=0$ where $S=\{ g_1,\ldots,g_r\}$.

Conversely, suppose that $\idm{S}{\chi}m=0$ for some finite set $S$. Then 
\[
m=(1-\idm{S}{\chi})\cdot m \in \upp{M}{\chi} 
\]
by Lemma \ref{lem:upp}
\end{proof}

\begin{cor}\label{cor:exact}
The functor $M\mapsto \dwn{M}{\chi}$ is an exact functor on the category of 
$\zhalf{\mathcal{R}}$-modules. \end{cor}

\begin{proof}
Since this is the functor $\dwn{(\mathcal{R})}{\chi}\otimes_{\mathcal{R}}(-)$ it is right-exact. 

On the other hand, if 
$N$ is an $\zhalf{\mathcal{R}}$-submodule of $M$ then Corollary \ref{cor:upp} implies immediately that 
$\upp{N}{\chi}=\upp{M}{\chi}\cap N$, from which left-exactness follows. 
\end{proof}

\begin{cor}\label{cor:exact1}
The functor $M\mapsto \upp{M}{\chi}$ is an exact functor on the category of 
$\zhalf{\mathcal{R}}$-modules. 

In particular, taking $\chi=\chi_0$, the functor $M\mapsto \aug{F}M$ is an exact functor 
on the category of 
$\zhalf{\mathcal{R}}$-modules.
\end{cor}
\begin{proof} This follows from Corollary \ref{cor:exact}, since for any $\chi\in\dual{G}$, 
 $\upp{M}{\chi}$ is the kernel of the surjective homomorphism $M\to \dwn{M}{\chi}$. 
\end{proof}

\begin{lem} \label{lem:zerodetect}
Let $M$ be a $\zhalf{\mathcal{R}}$-module. If $M\not=\{ 0\}$, there exists $\chi\in \dual{G}$ such that 
$\dwn{M}{\chi}\not=0$.
\end{lem}

\begin{proof} 

Since the functor $N\to \dwn{N}{\chi}$ is exact, it is enough to prove that $\dwn{N}{\chi}\not=0$ 
for some character $\chi$ and some submodule $N$ of $M$. 

Let $N$ be a non-zero cyclic submodule of $M$. So $N\cong \zhalf{\mathcal{R}}/J$ 
for some proper ideal 
$J$ of $\zhalf{\mathcal{R}}$. 

Let $\mathcal{M}$ be a maximal ideal of $\zhalf{\mathcal{R}}$ containing $J$. 
Let $k$ be the field 
$\zhalf{\mathcal{R}}/\mathcal{M}$ and let $\pi:\zhalf{\mathcal{R}}\to k$ 
be the canonical homomorphism. 

Let $\chi:G\to \mu_2$ be the character defined by $\chi(g)=\pi(g)\in \mu_2(k)=\mu_2$. 
Then $\upp{\zhalf{\mathcal{R}}}{\chi}\subset \mathcal{M}$. Hence $\dwn{k}{\chi}=k\not=\{0\}$. 

Since there is a surjective $\zhalf{\mathcal{R}}$-module homomorphism $N\to k$ it follows, 
by exactness, that 
$\dwn{N}{\chi}\not= \{ 0\}$.  
\end{proof}

Combining Corollary \ref{cor:exact} and Lemma \ref{lem:zerodetect} we deduce the following 
\emph{character-theoretic local-global principle}:

\begin{prop}\label{prop:localglobal} Let $f:M\to N$ be a homomorphism of 
$\zhalf{\mathcal{R}}$-modules. 

For any $\chi\in\dual{G}$, let $f_\chi$ be the induced homomorphism 
$\dwn{M}{\chi}\to\dwn{N}{\chi}$.

\begin{enumerate}
\item $f$ is injective if and only if $f_\chi$ is injective for all $\chi\in \dual{G}$.
\item $f$ is surjective if and only if $f_\chi$ is surjective for all $\chi\in \dual{G}$.
\item $f$ is an isomorphism if and only if $f_\chi$ is an isomorphism for all $\chi\in \dual{G}$.
\end{enumerate}
\end{prop}

\begin{proof} From exactness of the functor $M\mapsto \dwn{M}{\chi}$ it follows that
$
\ker{f_\chi}\cong \dwn{\ker{f}}{\chi}
$
and 
$
\coker{f_\chi}\cong \dwn{\coker{f}}{\chi}
$
for all $\chi$. Combining this with  Lemma \ref{lem:zerodetect} we deduce:

$\ker{f}=0$ if and only if $\ker{f_\chi}=0$ for all $\chi$ and 
$\coker{f}=0$ if and only if $\coker{f_\chi}=0$ for all $\chi$.   
\end{proof}

We will apply this principle several times below to the particular case $G=\sq{F}$ and 
$\mathcal{R}=\sgr{F}$. 

Combining Proposition \ref{prop:localglobal} with Corolllary \ref{cor:upph} we deduce:

\begin{cor}\label{cor:localglobalh}
Let $f:M\to N$ be a homomorphism of 
$\zhalf{\mathcal{R}}$-modules. Let $H$ be a finite subset of $\dual{G}$.

Then $f$ induces an isomorphism $\upp{M}{H}\to\upp{N}{H}$ if and only if 
$f_\chi:\dwn{M}{\chi}\to \dwn{N}{\chi}$ is an isomorphism for all 
$\chi\not\in H$. 
\end{cor}

In particular if we consider the particular case $H=\{ \chi_0\}$, we deduce
\begin{cor}\label{cor:localglobalaug}
Let $f:M\to N$ be a homomorphism of 
$\zhalf{\mathcal{R}}$-modules. 

Then $f$ induces an isomorphism $\aug{G}M\to\aug{G}N$ if and only if 
$f_\chi:\dwn{M}{\chi}\to \dwn{N}{\chi}$ is an isomorphism for all 
$\chi\not=\chi_0$.
\end{cor}

For $g\in G$, we will also consider the orthogonal idempotents
\[
\idm{+}{g}=\frac{1+g}{2},\quad \idm{-}{g}:= \frac{1-g}{2}\quad \mbox{ in } \zhalf{\mathcal{R}}.
\]

Since $\idm{+}{g}\in \upp{\zhalf{\mathcal{R}}}{\chi}$ if $\chi(g)=-1$ and 
$\idm{-}{g}\in \upp{\zhalf{\mathcal{R}}}{\chi}$ if $\chi(g)=1$ we easily deduce: 

\begin{lem} Let $M$ be a $\zhalf{\mathcal{R}}$-module. Let $g\in G$, $\chi\in \dual{G}$. Then
\[
\dwn{\left(\idm{+}{g}M\right)}{\chi}=
\left\{
\begin{array}{ll}
\idm{+}{g}M,& \chi(g)=1\\
0, & \chi(g)=-1
\end{array}
\right.
\quad
\mbox{ and }
\quad
\dwn{\left(\idm{-}{g}M\right)}{\chi}=
\left\{
\begin{array}{ll}
\idm{-}{g}M,& \chi(g)=-1\\
0, & \chi(g)=1
\end{array}
\right.
\]
\end{lem}
\section{The action of $\an{-1}$ on the refined scissors congruence group}
\label{sec:rsc}
In this section we prove (Theorem \ref{thm:rpbker}) 
that the square class $\an{-1}$ acts trivially on 
$\zhalf{\rpbker{F}}$ for any field $F$. We deduce that $\an{-1}$ acts trivially on 
$\ho{3}{\spl{2}{F}}{\zhalf{\Z}}$ and that $\zhalf{\rpbker{F}}$ is isomorphic to 
the $\sgr{F}$-module $\ep{-1}\zhalf{\qrpb{F}}$, and hence also to the 
quotient module $\zhalf{\qrpb{F}}/\epm{-1}\zhalf{\qrpb{F}}$ (Corollary 
\ref{cor:rpbker}). 

For any $\chi\in\dual{\sq{F}}$ and $\an{a}\in \sq{F}$, 
we will let $\dwn{\gpb{a}}{\chi}$ denote the image of $\gpb{a}$ in
$\dwn{\rrpb{F}}{\chi}$. 

\begin{lem}\label{lem:main1}
 Let $\chi\in\dual{\sq{F}}$. Then 
$2\dwn{\gpb{a}}{\chi}=0$ in $\dwn{\rrpb{F}}{\chi}$ if $\chi(\an{-a})=-1$.
\end{lem}
\begin{proof} By definition of $\dwn{M}{\chi}$, $(\an{-a}+1)x=0$ 
for all $x\in \dwn{\rrpb{F}}{\chi}$. 

On the other 
hand, $(\an{-a}-1)\gpb{a}=0$ in $\rrpb{F}$ by Lemma \ref{lem:pb-}. 

Thus $0=(1+\an{-a})\dwn{\gpb{a}}{\chi}+(1-\an{-a})\dwn{\gpb{a}}{\chi}=
2\dwn{\gpb{a}}{\chi}$.
\end{proof}

\begin{lem}\label{lem:ichi}
 Let $\chi\in \dual{\sq{F}}$. 

Let $a_1,\ldots, a_n\in F^\times$,  not necessarily distinct, 
satisfy $\chi(a_i)=-1$ for $i\leq n$. Then $\dwn{\zhalf{\aug{F}}}{\chi}$ 
is a free $\zhalf{\Z}$-module 
with generator $\prod_{i=1}^n\idm{-}{a_i}$.
\end{lem}

\begin{proof} Since $\chi\not=\chi_0$ by hypothesis, applying the exact functor $M\to \dwn{M}{\chi}$ 
to the sequence of $\zhalf{\sgr{F}}$-modules
\[
0\to \zhalf{\aug{F}}\to \zhalf{\sgr{F}}\to \zhalf{\Z}\to 0
\]
shows that the inclusion map $\aug{F}\to \sgr{F}$ induces an isomorphism 
$\dwn{\zhalf{\aug{F}}}{\chi}\cong \dwn{\zhalf{\sgr{F}}}{\chi}$. 

Of course, the homomorphism 
$\rho(\chi)$ induces a ring isomorphism  $\dwn{\zhalf{\sgr{F}}}{\chi}\cong \zhalf{\Z}$. 

Since $\idm{-}{a_i}\in \aug{F}$ and since $\rho(\chi)(\idm{-}{a_i})=1$ for all $i$, the 
statement of the Lemma follows. 
\end{proof}

Since $\aug{F}/\aug{F}^2\cong \sq{F}$, it follows that $\zhalf{\aug{F}^2}=\zhalf{\aug{F}}$. 

Furthermore, for any $\zhalf{\sgr{F}}$-module $M$, we have $\pp{-1}M=\ep{-1}M$ and hence 
\[
\frac{M}{\pp{-1}M}=\frac{M}{\ep{-1}M}\cong \epm{-1}M.
\]

Thus, tensoring the homomorphism
\[
\bar{\lambda}_1:\rrpb{F}\to \frac{\aug{F}^2}{\pp{-1}\aug{F}}
\]
with $\zhalf{\Z}$ gives a homomorphism 
\[
\bar{\lambda}_1:\zhalf{\rrpb{F}}\to \zhalf{\epm{-1}\aug{F}}
\]
whose kernel is $\zhalf{\rrpbker{F}}$.  In particular, it follows that 
\[
\ep{-1}\zhalf{\rrpb{F}}\subset \zhalf{\rrpbker{F}}. 
\] 

In fact we have: 
\begin{thm}\label{thm:rpbker}
For any field $F$
\[
\zhalf{\rrpbker{F}}= \ep{-1}\zhalf{\rrpb{F}}.
\]
\end{thm}

\begin{proof} We must prove that 
$
\ker{\bar{\lambda}_1}\cap \epm{-1}\zhalf{\rrpb{F}}=\{ 0\}
$. 
That is, we must prove that the induced map 
\[
\bar{\lambda}_1:\epm{-1}\zhalf{\rrpb{F}}\to \epm{-1}\aug{F}
\]
is injective. 

By the local-global principle, this is equivalent to the statement that the homomorphism
\[
\dwn{(\bar{\lambda}_1)}{\chi}:\dwn{\zhalf{\rrpb{F}}}{\chi}\to \dwn{\zhalf{\aug{F}}}{\chi}
\]
is injective for all characters $\chi$ satisfying $\chi(-1)=-1$.  

Let $\chi\in\dual{\sq{F}}$ satisfy $\chi(-1)=-1$. 

By Lemma \ref{lem:main1}, whenever $a\in F^\times$ with $\chi(a)=1$ then $\dwn{\gpb{a}}{\chi}=0$ 
in $\zhalf{\rrpb{F}}$. Since 
\[
\dwn{\gpb{a}}{\chi}=\dwn{\gpb{1-a}}{\chi}\mbox{ in } \dwn{\zhalf{\rrpb{F}}}{\chi}
\]
it follows that 
$
\dwn{\gpb{a}}{\chi}=0 \mbox{ whenever } \chi(a)=1 \mbox{ or } \chi(1-a)=1.
$

Thus, if it is the case that $\chi(a)=1$ or $\chi(1-a)=1$ for all $a\in F^\times\setminus\{ 1\}$ 
then 
$\dwn{\zhalf{\rrpb{F}}}{\chi}=0$ and there is nothing more to prove.

Otherwise, these exists $a\in F^\times$ such that 
\[
\chi(a)=-1=\chi(1-a).
\]

Suppose now that $a\not = b\in F^\times\setminus\{ 1\}$ satisfy 
$
\dwn{\gpb{a}}{\chi}\not=0,\  \dwn{\gpb{b}}{\chi}\not= 0 \mbox{ in } 
\dwn{\zhalf{\rrpb{F}}}{\chi}.
$

Then necessarily $\chi(a)=\chi(1-a)=\chi(b)=\chi(1-b)=-1$. 

Then in $\dwn{\zhalf{\rrpb{F}}}{\chi}$ we have 
\begin{eqnarray}\label{eq:dwn}
0= \dwn{\left( S_{a,b}\right)}{\chi} = 
\dwn{\gpb{a}}{\chi}-\dwn{\gpb{b}}{\chi}-\dwn{\gpb{\frac{b}{a}}}{\chi}
-\dwn{\gpb{\frac{1-a^{-1}}{1-b^{-1}}}}{\chi}-\dwn{\gpb{\frac{1-a}{1-b}}}{\chi}.
\end{eqnarray}

However 
\[
\chi\left(\frac{b}{a}\right)=1=\chi\left(\frac{1-a}{1-b}\right)
\]
and hence also 
\[
\chi\left(\frac{1-a^{-1}}{1-b^{-1}}\right)=\chi\left(\frac{b}{a}\cdot \frac{1-a}{1-b}\right)=1.
\]

So the last three terms in (\ref{eq:dwn}) are all $0$ by Lemma \ref{lem:main1}. 

Thus, for any $a,b\in F^\times$ such that $\dwn{\gpb{a}}{\chi}\not=0$ and 
$\dwn{\gpb{b}}{\chi}\not=0$ we have 
$\dwn{\gpb{a}}{\chi}=\dwn{\gpb{b}}{\chi}$ in $\dwn{\zhalf{\rrpb{F}}}{\chi}$. 

Thus, if there exists $a\in F^\times$ such that $\dwn{\gpb{a}}{\chi}\not=0$ then 
$\dwn{\zhalf{\rrpb{F}}}{\chi}$ is a cyclic $\zhalf{\Z}$-module with generator 
$\dwn{\gpb{a}}{\chi}$.

In this case, 
\[
\dwn{(\bar{\lambda}_1)}{\chi}\left( \dwn{\gpb{a}}{\chi}\right) = 4\epm{a}\epm{1-a}
\]
is a generator of the free $\zhalf{\Z}$-module $\dwn{\zhalf{\aug{F}}}{\chi}$ by Lemma 
\ref{lem:ichi}. It follows that the map $\dwn{(\bar{\lambda}_1)}{\chi}$ is 
an isomorphism, and the Theorem is proved.
\end{proof}

Since $\an{-1}$ acts trivially on $\cconst{F}$ it follows that 
$\cconstmod{F}=\ep{-1}\cconstmod{F}$ and we immediately deduce: 

\begin{cor}\label{cor:qrpbker} For any field $F$ 
\[\
\zhalf{\qrpbker{F}}=\ep{-1}\zhalf{\qrpb{F}}. 
\] 
\end{cor}

\begin{proof}
We begin by noting that if $M$ is a $\zhalf{\sgr{F}}$-module then 
$\an{-1}$ acts trivially on $M$ if and only if 
$M=\ep{-1}M$. 
Furthermore, any module $M$ decomposes as $\ep{-1}M\oplus \epm{-1}M$. It follows that 
functor $M\to \ep{-1}M$ is an exact functor on the category of $\zhalf{\sgr{F}}$-modules. 

The corollary thus follows from the short exact sequence of $\zhalf{\sgr{F}}$-modules
\[
0\to \cconstmod{F}\to \zhalf{\qrpbker{F}}\to \zhalf{\rrpbker{F}}\to 0.
\]
\end{proof}

Since $\ep{-1}\zhalf{\qrpb{F}}$ is a homomorphic image of $\zhalf{\qrpb{F}}$ and since 
$\zhalf{\rpbker{F}}\cong  \zhalf{\qrpbker{F}}$ by Corollary \ref{cor:qrbl}, it follows that 
$\zhalf{\rpbker{F}}$ can be described as an $\sgr{F}$-module by an explicit presentation:

Let $F$ be a field with at least $4$ elements. Let $\rsc{F}$ denote the $\sgr{F}$-module 
with the 
following presentation:

The generators are symbols $\grsc{x}$, $x\in F^\times$. 

These are subject to the following (families of) relations:
\begin{enumerate}
\item $\grsc{1}=0$
\item $ \grsc{x}-\grsc{y}+\an{x}\grsc{ \frac{y}{x}}-\an{x^{-1}-1}\grsc{\frac{1-x^{-1}}{1-y^{-1}}}
+\an{1-x}\grsc{\frac{1-x}{1-y}}=0$, for $x,y\not= 1$
\item $\an{-1}\grsc{x}=\grsc{x}$ for all $x$.
\item $\grsc{x}+\grsc{x^{-1}}=0$  for all $x$
\end{enumerate}

\begin{cor}\label{cor:pres} Let $F$ be a field with at least $4$ elements.
The $\sgr{F}$-module homomorphism
\[
\zhalf{\rsc{F}}\to \zhalf{\rpbker{F}},\quad 
\grsc{x}\mapsto \ep{-1}\gpb{x}+\frac{1}{2}\pf{1-x}\suss{1}{x}
\]
is an isomorphism. 
\end{cor} 
\begin{cor}\label{cor:-1}
For any field $F$, the square class  $\an{-1}$ acts trivially 
on $\ho{3}{\spl{2}{F}}{\zhalf{\Z}}$.  
\end{cor}

\begin{proof} 
By Corollary  \ref{cor:qrpbker}, $\an{-1}$ acts trivially  on  $\zhalf{\rpbker{F}}$, and hence also on the 
submodule $\zhalf{\rbl{F}}$. 

The action of $\sq{F}$ on the module  $\Tor{\mu_F}{\mu_F}$ is trivial. In particular 
$\an{-1}$ acts trivially on this module. The Corollary now follows from the short exact 
sequence 
\[
0\to \zhalf{\Tor{\mu_F}{\mu_F}}\to\ho{3}{\spl{2}{F}}{{\zhalf{\Z}}}\to\zhalf{{\rbl{F}}}\to 0
\]
of Theorem \ref{thm:main}. 
\end{proof}

\begin{rem} On the other hand, the square class $\an{-1}$ does not generally act trivially on 
$\hoz{2}{\spl{2}{F}}$. There is a natural isomorphism of $\sgr{F}$-modules 
$\hoz{2}{\spl{2}{F}}\cong \mwk{2}{F}$, where $\mwk{2}{F}$ is the second Milnor-Witt $K$-group of 
the field $F$ (see \cite[Appendix]{sus:tors}). 
Furthermore, there is  a natural surjection of $\sgr{F}$-modules $\mwk{2}{F}\to I^2(F)$ where 
$I^2(F)$ is the square of the fundamental ideal of 
the Witt ring $W(F)$ of $F$. 

For example, $I^2(\R)$ is the free $\Z$-module generated by $2\pf{-1}$  and $\an{-1}$ acts as 
multiplication by $-1$. Thus $\an{-1}$ acts non-trivally on $\hoz{2}{\spl{2}{\R}}$ and on 
$\ho{2}{\spl{2}{\R}}{\zhalf{\Z}}$. 
\end{rem}
\section{Valuations}\label{sec:val}
\subsection{General (Krull) valuations} 
Given a field $F$ and a totally ordered (additive) abelian group $\Gamma$, 
a valuation on $F$ with values 
in $\Gamma$ is a  group homomorphism
$v:F^\times\to \Gamma$ such that $v(a+b)\geq \mathrm{min}\{ v(a), v(b)\}$ whenever 
$a,b,a+b\in F^\times$. We also set $v(0):=\infty$ where $\gamma<\infty$ for all $\gamma\in \Gamma$.
By replacing $\Gamma$ with $v(F^\times)$ if ncessary, 
We may assume without loss that the homomorphism $v$ is surjective.

Given a valuation $v$, the corresponding valuation ring $O_v:=\{ a\in F\ |\ v(a)\geq 0\}$ is a local 
ring with maximal ideal $\mathcal{M}_v:= \{ a\in F \ |\ v(a)>0\}$ and group of 
units $U=U_v=\{ a\in F\ |\ v(a)=0\ \}$. We let $k=k_v$ be the associated residue field 
$O_v/\mathcal{M}_v$. The canonical ring homomorphism $O_v\to k$ induces a surjective group 
homomorphism $U_v\to k^\times$ whose kernel we denote $U_1=U_{1,v}$.  

Since $\Gamma$ is torsion-free, 
the short exact sequence of groups $1\to U\to F^\times\to \Gamma \to 0$ yields 
a (noncanonically split) short exact sequence 
\[
1\to U/U^2\to \sq{F}\to \Gamma/2\to 0.  
\]

Since $U_1\cap U^2=U_1^2$ there is also a (split) short exact sequence
\[
1\to U_1/U_1^2\to U/U^2\to \sq{k}\to 1.
\]
In particular, there is a natural surjective ring homomorphism $\Z[U/U^2]\to \sgr{k}$ 
which is an isomorphism when $U_1=U_1^2$. 

Given any $\sgr{k}$-module $M$, 
we define the induced $\sgr{F}$-module
\[
\indf{k}{F}{M}:=\sgr{F}\otimes_{\Z[U/U^2]}M.
\]


We recall (see  Appendix \ref{sec:appendix} below) 
that there is a natural surjective specialization or residue
homomorphism
\begin{eqnarray*}
S:=S_{v}:\rpb{F}&\to&\indf{k}{F}{\qrpb{k}}\\
\gpb{a}&\mapsto&\left\{
\begin{array}{rc}
1\otimes \gpb{\bar{a}},& v(a)=0\\
1\otimes\bconst{k},& v(a)>0\\
-(1\otimes\bconst{k}),& v(a)<0\\
\end{array}
\right.
\end{eqnarray*}

We will also let $\spec{v}$ denote the induced 
(surjective) specialization homomorphism on reduced scissors congruence groups
\[
\spec{v}:\rrpb{F}\to \indf{k}{F}{\rrpb{k}},\quad \gpb{a}\mapsto
\left\{
\begin{array}{ll}
1\otimes \gpb{\bar{a}},& v(a)=0\\
0,&v(a)\not=0
\end{array}
\right.
\]

We let $\kv{v}$ denote the $\sgr{F}$-submodule of $\rrpb{F}$ generated 
by set $\{ \gpb{a}\in \rrpb{F}\ |\ v(a)\not= 0\}$.

\begin{lem}\label{lem:lfv} There is a short exact sequence of $\sgr{F}$-modules
\[
\xymatrix{
0\ar[r]
&
\kv{v}\ar[r]
&
\rrpb{F}\ar[r]^-{\spec{v}}
&
\indf{k}{F}{\rrpb{k}}\ar[r]
&
0}
\]
\end{lem}
\begin{proof}
Clearly $\kv{v}\subset\ker{\spec{v}}$ and hence $\spec{v}$ induces a surjective homomorphism 
\[
\rrpb{F}_v:=\frac{\rrpb{F}}{\kv{v}}\to \indf{k}{F}{\rrpb{k}}.
\]

By definition, if $a\in F^\times$ with $v(a)\not=0$, then $\gpb{a}=0$ in $\rrpb{F}_v$. 
Thus if $u\in U_1$, then $\gpb{u}=-\an{-1}\gpb{1-u}=0$ in $\rrpb{F}_v$ since $v(1-u)>0$.  

Now suppose that $a\in U\setminus U_1$ and $u\in U_1$. Then 
\[
0=\gpb{a}-\gpb{au}+\an{a}\gpb{u}-\an{a^{-1}-1}\gpb{u\cdot\frac{1-a}{1-au}}+
\an{1-a}\gpb{\frac{1-a}{1-au}}
=\gpb{a}-\gpb{au} \mbox{ in }\rrpb{F}_v
\] 
since $1-a\equiv 1-au\pmod{U_1}$. Thus for $a\in U$ the element $\gpb{a}$ in $\rrpb{F}_v$ 
depends only on image, 
$\bar{a}$, of $a$ in $k^\times$.

Furthermore, the action of $\igr{U/U^2}$ on $\rrpb{F}_v$ descends to an 
$\sgr{k}$-module structure: 

Let $a\in U$ 
and $u\in U_1$. Then $\an{-a}\gpb{a}=\gpb{a}$ by Lemma \ref{lem:pb-}. 
However, $\gpb{a}=\gpb{au}$ in $\rrpb{F}_v$.
Thus, we also have $\an{-au}\gpb{a}=\gpb{a}$ in $\rrpb{F}_v$, from which it follows that 
\[
\an{u}\gpb{a}=\an{-a}\gpb{a}=\gpb{a}\mbox{ in }\rrpb{F}_v.
\]

It therefore follows that there is a well-defined $\sgr{k}$-module homomorphism 
\[
T_v:\rrpb{k}\to\rrpb{F}_v,\quad \gpb{\bar{a}}\mapsto \gpb{a}
\]
which extends to an $\sgr{F}$-module homomorphism
\[
T_{v,F}:\indf{k}{F}{\rrpb{k}}\to \rrpb{F}_v.
\]

Since $T_{v,F}$ is surjective by the work above, and since clearly $S_v\circ T_{v,F}=
\id{\indf{k}{F}{\rrpb{k}}}$, it follows 
that $S_v:\rrpb{F}_v\to\indf{k}{F}{\rrpb{k}}$ is an isomorphism with inverse $T_{v,F}$.
\end{proof}

\begin{cor}\label{cor:sv} Let $v:F^\times \to \Gamma$ be a valuation on the field $F$ with 
residue field $k$.

There is a natural short exact sequence 
\[
\xymatrix{
0\ar[r]
&\ep{-1}\zhalf{\kv{v}}\ar[r]
&\zhalf{\rrpbker{F}}\ar^-{\spec{v}}[r]
&\indf{k}{F}{\zhalf{\rrpbker{k}}}\ar[r] 
&0.
}
\]
\end{cor}

\begin{proof} If $M$ is any $\zhalf{\sgr{k}}$-module, then since 
$\ep{-1}\in\zhalf{\Z}[U_v]$ we have 
$\ep{-1}(\indf{k}{F}{M})= \indf{k}{F}{(\ep{-1}M)}$. 

By Theorem \ref{thm:rpbker}, taking the $\ep{-1}$-components 
of the terms in the exact sequence of Lemma \ref{lem:lfv} gives the result. 
\end{proof}

\subsection{Discrete valuations}
Given a  valuation $v:F^\times\to \Gamma$ if we choose a splitting $s:\Gamma/2\to \sq{F}$
we obtain a decomposition 
\[
\sgr{F}=\Z[\sq{F}] \cong \Z[(U/U^2)\times (\Gamma/2)] =\Z[U/U^2][\Gamma/2].
\]
Thus if $M$ is any $\sgr{k}$-module we have an isomorphism of $\sgr{k}$-modules
\[
\indf{k}{F}{M}\cong M\otimes_{\Z}\Z[\Gamma/2]\cong \oplus_{\gamma\in \Gamma/2}M.
\]

In particular, if $v:F^\times \to \Z$ is a discrete valuation, a splitting 
$s:\Z/2\to \sq{F}$ amounts to the choice of (the square class of) a uniformizer $\pi$. 
For an $\sgr{k}$-module $M$ this gives a split short exact sequence of 
$\sgr{k}$-modules 
\[
\xymatrix{
0\ar[r]&
M\ar[r]^-{i_0}
&\indf{k}{F}{M}\ar^-{\rho_\pi}[r]
&M\ar[r]
&0\\}
\]
where $i_0(m)=1\otimes m$ and
\[
\rho_\pi(\an{u\pi^s}\otimes m)=
\left\{
\begin{array}{ll}
\an{\bar{u}}m, & s\mbox{ odd }\\
0,& s\mbox{ even }. 
\end{array}
\right.
\]
If we choose the right inverse
\[
s_\pi:M\to\indf{k}{F}{M},\quad m\mapsto \pf{\pi}\otimes m
\]
this yields a left inverse $\rho_0:= i_0^{-1}\circ(\id{}-s_\pi\rho_\pi)$ 
of $i_0$ given by the formula
\[
\rho_0:\indf{k}{F}{M}\to M, \quad \an{u\pi^s}\otimes m\mapsto \an{\bar{u}}m
\]
(for \emph{all} $s\in \Z$, $u\in U$). 

It follows that $\rho_0(\pf{u\pi^s}\otimes m)= \pf{\bar{u}}m$ for all $u\in U$, $s\in \Z$, 
and hence there is a short exact sequence of $\sgr{k}$-modules
\[
\xymatrix{
0\ar[r]&
\aug{k}M\ar[r]^-{i_0}
&\aug{F}\left(\indf{k}{F}{M}\right)\ar^-{\rho_\pi}[r]
&M\ar[r]
&0\\}
\] 
which is split by $\rho_0$. We deduce:

\begin{lem}\label{lem:ifind}
Let $v:F^\times\to \Z$ be a discrete valuation on $F$ with residue field $k$. Let 
$M$ be an $\sgr{k}$-module. Then there is an isomorphism of $\sgr{k}$-modules
\[
\xymatrix{
\aug{F}\indf{k}{F}{M}\ar[r]^-{(\rho_0,\rho_\pi)}& \aug{k}M\oplus M.\\
} 
\]
\end{lem}

Recall that for any field $F$ the natural map $\hoz{3}{\spl{2}{F}}\to\rpbker{F}$ 
induces an isomorphism 
\[
\hot{F}{\zhalf{\Z}}\cong \aug{F}\ho{3}{\spl{2}{F}}{\zhalf{\Z}}\cong
\aug{F}\zhalf{\rpbker{F}}.
\]
Thus, if $v:F^\times\to \Z$ is a discrete valuation with residue field $k$ and if 
$\pi$ is a choice of uniformizer we obtain a sequence of 
homomorphisms of $\Z[U/U^2]$-modules 
\[
\xymatrix{
\hot{F}{\zhalf{\Z}}\ar[r]^-{\cong}
&\aug{F}\zhalf{\rpbker{F}}\ar[r]^-{\spec{v}}
&\aug{F}\indf{k}{F}{\left(\zhalf{\rpbker{k}}\right)}\\
\ar[r]^-{\cong}
&\aug{k}\zhalf{\rpbker{k}}\oplus\zhalf{\rpbker{k}}\ar[r]^-{\cong}
&\hot{k}{\zhalf{\Z}}\oplus \zhalf{\rpbker{k}}
}
\]
which thus induce natural surjective homorphisms of $\Z[U/U^2]$-modules
\[
\delta_0:\hot{F}{\zhalf{\Z}}\to \hot{k}{\zhalf{\Z}}
\]
 and 
\[
\delta_{\pi}:\hot{F}{\zhalf{\Z}}\to \zhalf{\rpbker{k}}
\]
both depending on the choice of uniformizer $\pi$. 

The remainder of this article concerns conditions under which the map $\spec{v}$,
 as it occurs here, is an isomorphism.  
\section{The main theorem and applications to $\ho{3}{\spl{2}{F}}{\Z}$}
\label{sec:main}
In this section we state and prove  our main theorem (Theorem \ref{thm:mainsv}) and derive 
applications to the calculation of $\ho{\spl{2}{F}}{\zhalf{\Z}}$ for certain fields $F$ 
with discrete valuation. 

\begin{thm}\label{thm:mainsv}
Let $F$ be a field with discrete valuation $v:F^\times \to \Z$. Suppose that there exists some 
$n\geq 1$ such that $U_n\subset U_1^2$. Then the homomorphism 
\[
\spec{v}:\zhalf{\rrpbker{F}}\to \indf{k}{F}{\zhalf{\rrpbker{k}}}
\]
induces an isomorphism 
\[
\aug{F}\zhalf{\rrpbker{F}}\cong \aug{F}\left(\indf{k}{F}{\zhalf{\rrpbker{k}}}\right).
\]
\end{thm}

\begin{proof}
In view of Corollary \ref{cor:sv} and the local-global principle, we must prove that
\[
\dwn{\zhalf{\kv{v}}}{\chi}=0
\]
whenever $\chi\in \dual{\sq{F}}$ satisfies, $\chi\not=\chi_0$ and $\chi(-1)=1$. 

Given the definition of $\kv{v}$, we must prove that $\gpb{a}_\chi=0$ in 
$\dwn{\zhalf{\rrpb{F}}}{\chi}$ for all such $\chi$ whenever $v(a)\not=0$. Since 
$\sgpb{a^{-1}}_\chi=-\gpb{a}_\chi$ in $\dwn{\zhalf{\rrpb{F}}}{\chi}$, 
it is enough to prove that $\gpb{a}_\chi=0$ for any such $\chi$ whenever $v(a)>0$. 

Suppose now that $\chi\not=\chi_0$ and $\chi(-1)=1$. We consider two cases:
\begin{enumerate}
\item Suppose first that $\chi(u)=1$ for all $u\in U_v$. 

Since $\chi\not=\chi_0$ and since 
$\Gamma =\Z$ we must have $\chi(\an{a})=(-1)^{v(a)}$ for all $a\in F^\times$. 

Thus, if $a\in F^\times$ with $v(a)$ odd then $\chi(\an{-a})=\chi(\an{a})=-1$ and hence 
$\gpb{a}_\chi=0$ in $\dwn{\zhalf{\rrpb{F}}}{\chi}$ by Lemma \ref{lem:main1}. 

Let $a\in F^\times$ with $v(a)=2k$, $k\geq 1$. Choose $\pi\in F^\times$ with $v(\pi)=1$. 
Thus there exists $u\in U_v$ with $a=u\pi^{2k}$. 

Let $x=\pi^{-(2k-1)}$ and let $y=u\pi$. Since $v(x), v(y)$ are both odd we have 
\[
\gpb{x}_\chi=0=\gpb{y}_\chi\mbox{ in } \dwn{\zhalf{\rrpb{F}}}{\chi}.
\] 

Now $a= y/x$. Thus in $\dwn{\zhalf{\rrpb{F}}}{\chi}$ we have 
\begin{eqnarray*}
0=\left( S_{x,y}\right)_\chi &= &\gpb{x}_\chi-\gpb{y}_\chi-\gpb{a}_\chi
-\gpb{\frac{1-x^{-1}}{1-y^{-1}}}_\chi-\gpb{\frac{1-x}{1-y}}_\chi\\
&=& -\left(\gpb{a}_\chi
+\gpb{\frac{1-x^{-1}}{1-y^{-1}}}_\chi+\gpb{\frac{1-x}{1-y}}_\chi\right).\\
\end{eqnarray*}

However, 
\[
\frac{1-x}{1-y}=\frac{1}{\pi^{2k-1}}\frac{\pi^{2k-1}-1}{1-u\pi}\quad \mbox{ and }\quad 
\frac{1-x^{-1}}{1-y^{-1}}= a\cdot \frac{1-x}{1-y}
\]
so that 
\[
v\left(\frac{1-x}{1-y}\right) \quad \mbox{ and } \quad v\left( \frac{1-x^{-1}}{1-y^{-1}}\right)
\]
are both odd and hence the last two terms in the bracketted expression are both zero.  Thus 
$\gpb{a}_\chi=0$ as required.

\item There exists $u\in U_v$ such that $\chi(u)=-1$.  

Since $U_n\subset U_1^2$ by hypothesis, we have $\chi(u)=1$ whenever $u\in U_n$. 
Let $n_0\geq 1$ be minimal such that $\chi(u)=1$ for all $u\in U_{n_0}$. 

Thus there exists $u\in U_{n_0-1}$ with $\chi(u)=-1$, where, for the purposes of 
 the present argument, we define  $U_0:=U$ if necessary. 

Since 
\[
\chi(1-u^{-1})= \chi(-(1-u)/u)=\chi(-1)\chi(u)\chi(1-u)=-\chi(1-u),
\]   
replacing $u$ by $u^{-1}$ if necessary, we can suppose that 
\[
\chi(u)=-1=\chi(1-u).
\]

Let $p:=1-u$. So $\chi(p)=-1$ and $v(p)\geq n_0-1\geq 0$.

Now let $a\in F^\times$ with $v(a)>0$. If $\chi(a)=-1$ then $\gpb{a}_\chi=0$ by Lemma 
\ref{lem:main1} and there is nothing to prove. So we can suppose that $\chi(a)=1$. 

In   $\dwn{\zhalf{\rrpb{F}}}{\chi}$ we have 
\[
0=\left(S_{p,ap}\right)_\chi = \gpb{p}_\chi-\gpb{ap}_\chi-\gpb{a}_\chi-\gpb{aw}_\chi-\gpb{w}_\chi
\]
where 
\[
w=\frac{1-p}{1-ap}.
\]

Now $\chi(p)=\chi(ap)=-1$ and hence $\gpb{p}_\chi=\gpb{ap}_\chi=0$ by Lemma \ref{lem:main1}. 

Furthermore $v(ap)\geq n_0$ and hence $1-ap\in U_{n_0}$ so that $\chi(1-ap)=1$. On the 
other hand, $\chi(1-p)=\chi(u)=-1$. Thus $\chi(w)=\chi(aw)=1$. Therefore, 
$\gpb{w}_\chi=\gpb{aw}_\chi=0$ by Lemma \ref{lem:main1} and $\gpb{a}_\chi=0$ as required. 
\end{enumerate} 
\end{proof}

\begin{rem}
On the other hand, under the hypotheses of the theorem the full homomorphism
\[
\spec{v}:\zhalf{\rrpbker{F}}\to \indf{k}{F}{\zhalf{\rrpbker{k}}}
\] 
is almost never an isomorphism since 
\[
\dwn{\left(\zhalf{\rrpbker{F}}\right)}{\chi_0}\cong \zhalf{\redpb{F}}\mbox{ while }
\dwn{\left(\indf{k}{F}{\zhalf{\rrpbker{k}}}\right)}{\chi_0}\cong \zhalf{\redpb{k}}. 
\]
\end{rem}

\begin{rem} Let $F$ be a field with discrete valuation $v:F^\times \to \Z$  and 
residue field $k$.
If $F$ is complete with respect to the valuation $v$ and if the residue characteristic is not $2$, 
then in fact we have $U_1=U_1^2$ 
(see, for example, \cite[Chapter I, Corollary 5.5]{fesenko:vostokov}), 
so that the hypothesis of Theorem \ref{thm:mainsv} holds.

Let $F$ be a field of characteristic $0$ with discrete valuation $v:F^\times \to \Z$  and 
finite residue field $k$ of characteristic $2$.
If $F$ is complete with respect to the valuation $v$ then there exists $n\geq 1$ such that 
$U_n\subset U_1^2$. (See, for example, \cite[Chapter II, Proposition 5.7 (i)]{neukirch:ant}.) 
For example, when $F=\Q_2$, $U_1^2=U_3$. 

On the other hand, the hypothesis of the theorem 
fails for the field $F=\laurs{\F{2}}{t}$, where $U_1^2$ has infinite 
index in $U_1$ (\cite[Chapter II, Proposition 5.7 (ii)]{neukirch:ant}). 

However, the theorem is certainly not true without some strong condition on the units:

 For example, 
if we take $F=\Q$ and let $v=v_p$ for some prime $p$, then 
$\aug{\Q}\zhalf{\rrpbker{\Q}}$ is an infinite torsion abelian group while 
$\zhalf{\rrpbker{\F{p}}}=\zhalf{\pb{\F{p}}}$ is a finite cyclic group.  
\end{rem}

Let $F$ be  a field with discrete valuation $v:F^\times\to \Z$ and residue field $k$. 
Let $\pi\in F^\times$ satisfy $v(\pi)=1$. Let $M$ be a $\sgr{k}$-module. 
In Lemma \ref{lem:ind} above, we described a decomposition 
$\indf{k}{F}{M}\cong M\oplus M$ as a $\sgr{k}$-module.  We now describe  decompositions 
of  the modules $\indf{k}{F}{\zhalf{M}}$ and $\aug{F}\left(\indf{k}{F}{\zhalf{M}}\right)$
as $\sgr{F}$-modules. 

 Note that  if $a\in F^\times$ there exists a unique 
$u_\pi(a)\in U$ satisfying $a=u_\pi(a)\pi^{v(a)}$. 
For any ${\Z}[U/U^2]$-module $M$ we will 
write $\pieps{M}{\pi}{1}$ and $\pieps{M}{\pi}{-1}$ for the 
module $M$ equipped with the $\sgr{F}$-module structure 
\[
\an{a}m:= \an{u_\pi(a)}m \quad \mbox{ and }\quad \an{a}m:= (-1)^{v(a)}\an{u_\pi(a)}m  
\]
respectively.

\begin{lem}\label{lem:pieps}
 Let $F$ be  a field with discrete valuation $v:F^\times\to \Z$ and residue field $k$. 
Let $\pi\in F^\times$ satisfy $v(\pi)=1$. 

Let $M$ be a $\zhalf{\Z}[U/U^2]$-module. Let 
$\chi\in\dual{\sq{F}}$. Let $\epsilon\in \{ 1,-1\}$. Then 
\begin{enumerate}
\item 
\[
\dwn{\left(\pieps{M}{\pi}{\epsilon}\right)}{\chi}=
\left\{
\begin{array}{ll}
\pieps{\left(\dwn{M}{\chi|_U}\right)}{\pi}{\epsilon}, & \chi(\pi)=\epsilon\\
0,& \chi(\pi)=-\epsilon.
\end{array}
\right.
\]

\item
\[
\upp{\pieps{M}{\pi}{\epsilon}}{\chi}=
\left\{
\begin{array}{ll}
\pieps{\left(\upp{M}{\chi|_U}\right)}{\pi}{\epsilon}, & \chi(\pi)=\epsilon\\
\pieps{M}{\pi}{\epsilon},& \chi(\pi)=-\epsilon.
\end{array}
\right.
\]

\end{enumerate}
\end{lem}

\begin{proof}\ 

\begin{enumerate}
\item If $\chi(\pi)=-\epsilon$ then the square class $\an{\pi}$ acts as multiplication by $1$ and $-1$ 
on $\dwn{\pieps{M}{\pi}{\epsilon}}{\chi}$, so that this module is zero.

Suppose, on the other hand,  that $\chi(\pi)=\epsilon$. Then the homomorphism of $\Z[U/U^2]$-modules 
$M\to \dwn{M}{\chi|_U}$ induces a homomorhism of $\sgr{F}$-modules 
$\pieps{M}{\pi}{\epsilon}\to \pieps{\left(\dwn{M}{\chi|_U}\right)}{\pi}{\epsilon}$ which thus 
induces a $\sgr{F}$-homomorphism 
\[
\dwn{\pieps{M}{\pi}{\epsilon}}{\chi}\to \pieps{\left(\dwn{M}{\chi|_U}\right)}{\pi}{\epsilon}.
\]

On the other hand, the $\Z[U/U^2]$-homomorphism 
$\dwn{M}{\chi|_U}\to \dwn{\pieps{M}{\pi}{\epsilon}}{\chi}$ gives an inverse to this map.
\item This follows immediately from (1) and the exactness of the sequence 
\[
0\to \upp{\pieps{M}{\pi}{\epsilon}}{\chi}\to 
\pieps{M}{\pi}{\epsilon} \to \dwn{\pieps{M}{\pi}{\epsilon}}{\chi}\to 0.
\]
\end{enumerate}
\end{proof}
\begin{lem}\label{lem:ikmf}
Let $F$ be  a field with discrete valuation $v:F^\times\to \Z$ and residue field $k$.
Let $\pi\in F^\times$ satisfy $v(\pi)=1$.
 
Let $M$ be a $\zhalf{\sgr{k}}$-module. Then there are natural decompositions of 
$\zhalf{\sgr{F}}$-modules
\[
\indf{k}{F}{M}\cong \pieps{M}{\pi}{1}\oplus \pieps{M}{\pi}{-1} 
\]
and 
\[
\aug{F}(\indf{k}{F}{M})\cong \pieps{(\aug{k}{M})}{\pi}{1}\oplus \pieps{M}{\pi}{-1}.
\]

In other words,  $\aug{F}(\indf{k}{F}{M})= \aug{k}M\oplus M$ made into an 
$\sgr{F}$-module by letting  $\an{\pi}$ act as the identity on the first
factor and as multiplication by $-1$ on the second factor.
\end{lem}

\begin{proof}
Since the natural short exact sequence of groups
\[
1\to U/U^2\to \sq{F}\to \Z/2\to 0.  
\]
is split the map $1\mapsto \an{\pi}$, we have a decomposition 
\[
\sq{F}\cong U/U^2\times C_\pi 
\]
where $C_\pi$ is the cyclic group of order $2$ with generator $\an{\pi}$. 
Thus 
\[
\sgr{F}=\Z[\sq{F}]\cong (\Z[U/U^2])[C_\pi].
\]

Hence 
\[
\zhalf{\sgr{F}}\cong \ep{\pi}\zhalf{\Z}[U/U^2]\times \epm{\pi}\zhalf{\Z}[U/U^2] 
\]
as a $\zhalf{\Z}[U/U^2]$-algebra. 

The first statement of the lemma follows at once.

The second statement of the lemma now follows from Lemma \ref{lem:pieps} (2), 
since $\aug{F}N=\upp{N}{\chi_0}$ for any $\sgr{F}$-module $N$.
\end{proof}

Applying this to the statement of Theorem \ref{thm:mainsv}, we deduce
\begin{cor}\label{cor:ikmf}
Let $F$ be a field with discrete valuation $v:F^\times \to \Z$ and residue field $k$.  
Suppose that there exists some 
$n\geq 1$ such that $U_n\subset U_1^2$. Let $\pi\in F^\times$ satisfy 
$v(\pi)=1$.

Then there is an isomorphism of $\zhalf{\sgr{F}}$-modules 
\[
\aug{F}\zhalf{\rrpbker{F}}\cong \aug{k}\zhalf{\rrpbker{k}}\oplus \zhalf{\rrpbker{k}} 
\] 
where $\an{\pi}$ acts as the identity on the first
factor and as multiplication by $-1$ on the second factor. 
\end{cor}

\begin{cor}\label{cor:ifrpf}
Let $F$ be a field with discrete valuation $v:F^\times \to \Z$  and residue field $k$.
  Suppose that there exists some 
$n\geq 1$ such that $U_n\subset U_1^2$. 

Let $\chi_v:\sq{F}\to \mu_2$ be the character $\chi_v(\an{a}):= (-1)^{v(a)}$. Let 
$W=W_v:=\{ \chi_0,\chi_v\}$ be the subgroup of $\dual{\sq{F}}$ generated by $\chi_v$. 

Let $\pi\in F^\times$ satisfy 
$v(\pi)=1$. 

Then there are natural isomorphisms of $\sgr{F}$-modules
\begin{eqnarray*}
\upp{\zhalf{\rrpbker{F}}}{W}&\cong &\left(\aug{k}\zhalf{\rrpbker{k}}\right)_F\\
&\cong & \pieps{\aug{k}\zhalf{\rrpbker{k}}}{\pi}{1}\oplus 
\pieps{\aug{k}\zhalf{\rrpbker{k}}}{\pi}{-1}.\\
\end{eqnarray*}
and
\[
\dwn{\zhalf{\rrpbker{F}}}{\chi_v}\cong \pieps{\left(\zhalf{\redpb{k}}\right)}{\pi}{-1}.
\]
\end{cor}

\begin{proof}
We have 
\begin{eqnarray*}
\upp{\zhalf{\rrpbker{F}}}{W}&=& \upp{\left(\upp{\zhalf{\rrpbker{F}}}{\chi_0}\right)}{\chi_v}
= \upp{\left(\aug{F}\zhalf{\rrpbker{F}}\right)}{\chi_v}\\
&\cong &  \upp{\left(\pieps{\aug{k}\zhalf{\rrpbker{k}}}{\pi}{1}\right)}{\chi_v}
\oplus \upp{\left(\pieps{\zhalf{\rrpbker{k}}}{\pi}{-1}\right)}{\chi_v}
\mbox{ by \ref{cor:ikmf}}\\ 
&\cong & \pieps{\aug{k}\zhalf{\rrpbker{k}}}{\pi}{1}\oplus 
\pieps{\aug{k}\zhalf{\rrpbker{k}}}{\pi}{-1}\mbox{ by \ref{lem:pieps} (2)}\\
&\cong &\indf{k}{F}{\left(\aug{k}\zhalf{\rrpbker{k}}\right)}.\\
\end{eqnarray*}

On the other hand
\begin{eqnarray*}
\dwn{\zhalf{\rrpbker{F}}}{\chi_v}&=& \dwn{\left(\upp{\zhalf{\rrpbker{F}}}{\chi_0}\right)}{\chi_v}
\mbox{ by \ref{cor:upph}}\\
&=& \dwn{\left(\aug{F}\zhalf{\rrpbker{F}}\right)}{\chi_v}\\
&\cong & \dwn{\left(\pieps{\aug{k}\zhalf{\rrpbker{k}}}{\pi}{1}\right)}{\chi_v}\oplus 
\dwn{\left(\pieps{\zhalf{\rrpbker{k}}}{\pi}{-1}\right)}{\chi_v}
\mbox{ by \ref{cor:ikmf}} \\
&\cong & \{ 0\}\oplus \pieps{\left(\dwn{\zhalf{\rrpbker{k}}}{\chi_0}\right)}{\pi}{-1}
\mbox{ by \ref{lem:pieps}(1)}\\
&=& \pieps{\left(\zhalf{\redpb{k}}\right)}{\pi}{-1}
\mbox{ by \ref{lem:qblochinf}}.
\end{eqnarray*} 
\end{proof}

Combining this with Corollary \ref{cor:rpbker3} above and the fact that
$\dwn{\ho{3}{\spl{2}{F}}{\zhalf{\Z}}}{\chi_o}\cong\zhalf{\kind{F}}$ (by Lemma \ref{lem:blochinf} 
(2)), we immediately deduce:

\begin{prop}\label{prop:hot6}
Let $F$ be a field with discrete valuation $v:F^\times \to \Z$  and residue field $k$.
  Suppose that there exists some 
$n\geq 1$ such that $U_n\subset U_1^2$. 

Let $\chi_v:\sq{F}\to \mu_2$ be the character $\chi_v(\an{a}):= (-1)^{v(a)}$. Let 
$W=W_v:=\{ \chi_0,\chi_v\}$ be the subgroup of $\dual{\sq{F}}$ generated by $\chi_v$. 

Let $\pi\in F^\times$ satisfy 
$v(\pi)=1$. 

Then there are natural isomorphisms of $\sgr{F}$-modules
\begin{eqnarray*}
\upp{\ho{3}{\spl{2}{F}}{\znth{\Z}{6}}}{W}&\cong &\indf{k}{F}{\left(\hot{k}{\znth{\Z}{6}}\right)}\\
\end{eqnarray*}
and
\[
\dwn{\ho{3}{\spl{2}{F}}{\znth{\Z}{6}}}{W}\cong \znth{\kind{F}}{6}\oplus 
\znth{\pb{k}}{6}.
\]

Thus there is a natural short exact sequence 
\[
0\to \left(\hot{k}{\znth{\Z}{6}}\right)^{\oplus 2} \to 
\ho{3}{\spl{2}{F}}{\znth{\Z}{6}}\to \znth{\kind{F}}{6}\oplus 
\znth{\pb{k}}{6}\to 0. 
\]

Furthermore, if $X^2-X+1$ has a root in $F$, 
we can replace $\frac{1}{6}$ with $\frac{1}{2}$ in these 
statements. 
\end{prop}

We now generalize Proposition \ref{prop:hot6} to  the following situation:  

Let $n\geq 1$ and let $F=F_0, F_1,\ldots, F_n$ be a sequence of fields such that 
for $i=1,\ldots, n$ there exists a discrete valuation $v_i$ on $F_{i-1}$ with residue field 
$F_i$. We suppose also that each of these valuations satisfies the condition that there exists 
$n_i\geq 1$ with $U_{n_i,v_i}\subset U_{1,v_i}^2$.

\begin{exa} 
In general  if $v_i$ is complete and $\chara{F_{i}}\not=2$ then the 
 stronger condition $U_{1,v_i}^2=U_{1,v_i}$ is 
satisfied. 

Thus for example, we can take any field $F_n$ of characteristic not equal to $2$ and let $F=F_0$ 
be the iterated Laurent series ring 
\[
F= \laurs{\laurs{F_n}{x_n}\cdots}{x_1}.
\]
Here $F_{i-1}= \laurs{\laurs{F_n}{x_n}\cdots}{x_i}$ and $v_i$ is the $x_i$-adic valuation. 

More generally, examples of this situation are provided by higher-dimensional local fields, which 
arise naturally in arithmetic geometry and higher-dimensional local class field theory 
(see, for example, \cite{parshin:lcft} or \cite{kato:classfield}). 
\end{exa}

Given this general set-up, we define inductively subgroups $\tilde{U}_i, i=0,\ldots, n$ of $F^\times$ by
\[
\tilde{U}_0:= F^\times,\quad  \tilde{U}_{i+1}=\{ u\in \tilde{U}_i\ |\ v_{i+1}(\bar{u})=0\},\quad 
 i=0, 1,\ldots, n-1
\]
and we let $\bar{U}_i$ denote the image of the homomorphism $\tilde{U}_i\to \sq{F}$. 
For each $i=1,\ldots, n$, we choose $\pi_i\in \tilde{U}_{i-1}$ satisfying $v_{i}(\overline{\pi_i})=1$. 
Thus 
\[
\bar{U}_{i-1}=\bar{U}_i\times C_{\pi_i}
\]
for $i=1,\ldots, n$. Hence 
\[
\sq{F}=\bar{U}_0\cong \bar{U}_n\times C_{\pi_n}\times \cdots \times C_{\pi_1}. 
\] 

For $i=0,\ldots, n$, we 
let $W_i$ denote the group $\{ \chi\in \dual{\sq{F}}\  |\ \chi|_{\bar{U}_i}=1\}$; ie. 
$W_i={\bar{U}_i}^\perp$.

Given $\chi\in W_n$ we define 
\[
i(\chi):= \mathrm{min}\{ i\leq n\ |\ \chi\in W_i\} = \mathrm{min}\{ i\leq n\ |\ 
\chi|_{\bar{U}_i}=1\}. 
\]

If $\underline{\epsilon}= \epsilon_i,\ldots,\epsilon_1$ is an ordered list of elements of 
$\mu_2$ and if $M$ is a $\Z[\bar{U}_i]$-module, we let $M[\underline{\epsilon}]$ denote 
the $\sgr{F}$-module 
\[
M[\underline{\epsilon}]=M[\epsilon_i,\ldots,\epsilon_1]:=
\pieps{\left((\pieps{M}{\pi_i}{\epsilon_i})\cdots\right)}{\pi_1}{\epsilon_1}.
\]

Furthermore, we let $\chi_{\underline{\epsilon}}$ be the unique character in $W_i$ 
satisfying
\[
\chi_{\underline{\epsilon}}(\pi_j)=\epsilon_j\mbox{ for } 1\leq j\leq i.
\]

Conversely, for $\chi\in W_n$, we define
\[
\underline{\epsilon}(\chi):= \chi(\pi_i),\ldots, \chi(\pi_1)\mbox{ where } i=i(\chi). 
\]

Using induction on $n$, Lemma \ref{lem:pieps} immediately implies:

\begin{lem}\label{lem:piepsn}
Let $i\geq 1$.  Let $M$ be a $\zhalf{\Z}[\bar{U}_i]$-module and let 
$\underline{\epsilon}=\epsilon_i,\ldots,\epsilon_1$ be an ordered list of elements of 
$\mu_2$. 

Then for any $\chi\in\dual{\sq{F}}$  
\begin{enumerate}
\item 
\[
\dwn{\left(M[\underline{\epsilon}]\right)}{\chi}=
\left\{
\begin{array}{ll}
\left(\dwn{M}{\chi|_{\bar{U}_i}}\right)[\underline{\epsilon}], & \chi(\pi_j)=\epsilon_j
\mbox{ for all } j\\
0,& \chi(\pi_j)=-\epsilon_j \mbox{ for some }j.
\end{array}
\right.
\]

\item
\[
\upp{{M}[\underline{\epsilon}]}{\chi}=
\left\{
\begin{array}{ll}
\left(\upp{M}{\chi|_{\bar{U}_i}}\right)[\underline{\epsilon}], & 
\chi(\pi_j)=\epsilon_j\mbox{ for all }j\\
M[\underline{\epsilon}],& \chi(\pi_j)=-\epsilon_j\mbox{ for some }j.
\end{array}
\right.
\]

\end{enumerate}

\end{lem}

If $M$ is a $\Z[\bar{U}_i]$-module for $i\geq 1$, we will let $\indf{F_i}{F}{M}$ denote the 
induced $\sgr{F}$-module
\[
\indf{F_i}{F}{M}=\sgr{F}\otimes_{\Z[\bar{U}_i]}M.
\]

When $M$ is a $\Z[\bar{U}_1]$ this is consistent with our use of this notation above.

Furthermore, for $1\leq j<i$ we have 
\[
\indf{F_i}{F}{M}
=\sgr{F}\otimes_{\Z[\bar{U}_j]}\left(\Z[\bar{U}_j]\otimes_{\Z[\bar{U}_i]}M\right)=
\sgr{F}\otimes_{\Z[\bar{U}_j]}\left(\sgr{F_j}\otimes_{\Z[\bar{U}_i]}M\right)
=\indf{F_j}{F}{\left(\indf{F_i}{F_j}{M}\right)}.
\]

Lemma \ref{lem:ikmf} together with induction on $n$ allows us to deduce:

\begin{lem}\label{lem:ikmfn}
Let $M$ be a $\zhalf{\Z}[\bar{U}_n]$-module. Then there is an isomorphism of 
$\sgr{F}$-modules
\[
\indf{F_n}{F}{M}\cong \oplus_{\underline{\epsilon}\in \mu_2^n}M[\underline{\epsilon}]. 
\]
\end{lem}

Given a finite subset $H$ of $\dual{\sq{F}}$ and a given
 $\underline{\epsilon}=\epsilon_i,\ldots,\epsilon_1\in \mu_2^i$ for 
some $i\leq n$, we will set 
\[
H_{\underline{\epsilon}}=\{ \chi\in H\ |\ \chi(\pi_j)=\epsilon_j \mbox{ for } 1\leq j\leq i\}.
\]

The following Lemma is an immediate consequence of Lemma \ref{lem:pieps} (2):
\begin{lem}\label{lem:heps} 
Let $1\leq i\leq n$ and let $M$ be a $\zhalf{\Z}[\bar{U}_i]$-module. 
Let $\underline{\epsilon}\in \mu_2^i$ and let $H$ be a finite subset of $\dual{\sq{F}}$.

Then 
\[
\upp{M[\underline{\epsilon}]}{H}= 
\upp{M[\underline{\epsilon}]}{H_{\underline{\epsilon}}}.
\]
\end{lem}

We thus obtain the following generalization of Corollary \ref{cor:ifrpf}:

\begin{prop}\label{prop:rpbwn}
Let $W=W_n$. 
We have
\[
\upp{\zhalf{\rrpbker{F}}}{W}\cong \indf{F_n}{F}{\left(\aug{F_n}\zhalf{\rrpbker{F_n}}\right)}
\]
and for any $\chi\in W\setminus\{ \chi_0\}$, there is an isomorphism of 
$\sgr{F}$-modules
\[
\dwn{\zhalf{\rrpbker{F}}}{\chi}\cong \zhalf{\redpb{F_i}}[\underline{\epsilon}(\chi)]
\]
where $i=i(\chi)$.
\end{prop}

\begin{proof} We will use induction on $n$. The case $n=1$ is precisely Corollary 
\ref{cor:ifrpf}. 

Suppose therefore that $n>1$ and the result is established for $n-1$. 

\begin{eqnarray*}
\upp{\zhalf{\rrpbker{F}}}{W_n} &=& \upp{\left(\upp{\zhalf{\rrpbker{F}}}{W_{n-1}}\right)}{W_n}\\
&\cong & \upp{\left(\aug{F_{n-1}}\zhalf{\rrpbker{F_{n-1}}}_{F}\right)}{W_n} \mbox{ by the 
inductive hypothesis}\\
&\cong & \upp{\left(\oplus_{\underline{\epsilon}\in \mu_2^{n-1}}
\zhalf{\rrpbker{F_{n-1}}}[\underline{\epsilon}]\right)}{W_n}\\
&=&\bigoplus_{\underline{\epsilon}\in \mu_2^{n-1}}
\left(\upp{\left(\zhalf{\rrpbker{F_{n-1}}}[\underline{\epsilon}]\right)}
{W_{n,\underline{\epsilon}}}\right).
\end{eqnarray*}

Now for any $\underline{\epsilon}\in \mu_2^{n-1}$, 
$W_{n,\underline{\epsilon}}=\{ \chi_{1,\underline{\epsilon}},\chi_{-1,\underline{\epsilon}}\}$ 
and 
\[
\chi_{1,\underline{\epsilon}}|_{\bar{U}_n}=\chi_0,\quad \chi_{-1,\underline{\epsilon}}|_{\bar{U}_n}=\chi_{v_n}
\]
so that
\[
(W_{n,\underline{\epsilon}})|_{\bar{U}_n}=\{ \chi_0,\chi_{v_n}\}= W_{v_n}.
\]

Thus 
\begin{eqnarray*}
\upp{\zhalf{\rrpbker{F}}}{W_n} &\cong&\bigoplus_{\underline{\epsilon}\in \mu_2^{n-1}}
\left(\upp{\left(\zhalf{\rrpbker{F_{n-1}}}[\underline{\epsilon}]\right)}
{W_{n,\underline{\epsilon}}}\right)\\
&=& \bigoplus_{\underline{\epsilon}\in \mu_2^{n-1}}
\left(\left(\upp{\zhalf{\rrpbker{F_{n-1}}}}
{W_{v_n}}\right)[\underline{\epsilon}]\right)\\
&\cong&\bigoplus_{\underline{\epsilon}\in \mu_2^{n-1}}
\left(\indf{F_n}{F_{n-1}}{\left(\aug{F_n}\zhalf{\rrpbker{F_n}}\right)}[\underline{\epsilon}]\right)
\mbox{ by Corollary \ref{cor:ifrpf}}\\
&\cong & \indf{F_{n-1}}{F}{\left(\indf{F_n}{F_{n-1}}{\left(\aug{F_n}\zhalf{\rrpbker{F_n}}\right)}
\right)}\\
&=& \indf{F_n}{F}{\left(\aug{F_n}\zhalf{\rrpbker{F_n}}\right)}.\\
\end{eqnarray*}

For the second statement, we use induction on $i(\chi)\geq 1$.

 If $i(\chi)=1$ then 
necessarily $\chi=\chi_{v_1}$ and the result is just Corollary \ref{cor:ifrpf} again. 

Suppose now that $i(\chi)>1$ and the result is known for characters $\psi$ with 
$i(\psi)<i(\chi)$. 

Since $\chi\not= \chi_0$ 
\begin{eqnarray*}
\dwn{\zhalf{\rrpbker{F}}}{\chi}&\cong& 
\dwn{\left(\zhalf{\upp{\rrpbker{F}}{\chi_0}}\right)}{\chi}
=\dwn{\left(\zhalf{\aug{F}\rrpbker{F}}\right)}{\chi}.\\
\end{eqnarray*}

By Corollary \ref{cor:ikmf} we thus have
\begin{eqnarray*}
\dwn{\zhalf{\rrpbker{F}}}{\chi}&\cong& 
\dwn{\left( \pieps{\aug{F_1}\zhalf{\rrpbker{F_1}}}{\pi_1}{1}\right)}{\chi}
\oplus \dwn{\left( \pieps{\zhalf{\rrpbker{F_1}}}{\pi_1}{-1}\right)}{\chi}\\
&\cong &
\left\{
\begin{array}{ll}
\pieps{\left(\dwn{\aug{F_1}\zhalf{\rrpbker{F_1}}}{\chi|_{\bar{U}_1}}\right)}{\pi_1}{1},
&\mbox{ if }\chi(\pi_1)=1\\
\pieps{\left(\dwn{\zhalf{\rrpbker{F_1}}}{\chi|_{\bar{U}_1}}\right)}{\pi_1}{-1}
,&\mbox{ if }\chi(\pi_1)=-1\\
\end{array}
\right.
\end{eqnarray*}

However, since $i(\chi)>1$, $\chi|_{\bar{U}_1}\not=\chi_0$ and thus 
\[
\dwn{\aug{F_1}\zhalf{\rrpbker{F_1}}}{\chi|_{\bar{U}_1}}=
\dwn{\zhalf{\rrpbker{F_1}}}{\chi|_{\bar{U}_1}}.
\]

Furthemore, observe that $i(\chi|_{\bar{U}_1})=i(\chi)-1$.

It follows that 
\begin{eqnarray*}
\dwn{\zhalf{\rrpbker{F}}}{\chi}&\cong& 
\pieps{\left(\dwn{\zhalf{\rrpbker{F_1}}}{\chi|_{\bar{U}_1}}\right)}{\pi_1}{\chi(\pi_1)}\\
&\cong & \pieps{\left( \zhalf{\redpb{F_i}}[\epsilon(\chi|_{\bar{U}_1})]\right)}{\pi_1}{\chi(\pi_1)}
\mbox{ by induction }\\
&=& \zhalf{\redpb{F_i}}[\epsilon(\chi)]
\end{eqnarray*}
as required. 
\end{proof}

\begin{thm}\label{thm:rpbwn} 
Let $n\geq 1$ and let $F=F_0, F_1,\ldots, F_n$ be a sequence of fields such that 
for $i=1,\ldots, n$ there exists a discrete valuation $v_i$ on $F_{i-1}$ with residue field 
$F_i$. We suppose also that each of these valuations satisfies the condition that there exists 
$n_i\geq 1$ with $U_{n_i,v_i}\subset U_{1,v_i}^2$.

Then there is a natural short exact sequence
\[
0\to 
\indf{F_n}{F}{\left(\hot{F_n}{\znth{\Z}{6}}\right)}\to 
\ho{3}{\spl{2}{F}}{\znth{\Z}{6}}\to
\znth{\kind{F}}{6}\oplus \left(\bigoplus_{i=1}^{n}\znth{\pb{F_i}}{6}^{\oplus 2^{i-1}}\right)
\to 0. 
\]

Furthermore, if $X^2-X+1$ has a root in $F$, we can 
replace $\frac{1}{6}$ with $\frac{1}{2}$ in this
statement.
\end{thm}

\begin{proof}
Since $\ho{3}{\spl{2}{F}}{\znth{\Z}{6}}$ is a $\znth{\sgr{F}}{6}$-module, there is a natural 
short exact sequence
\[
0\to 
\upp{\ho{3}{\spl{2}{F}}{\znth{\Z}{6}}}{W}\to 
\ho{3}{\spl{2}{F}}{\znth{\Z}{6}}\to
\dwn{\ho{3}{\spl{2}{F}}{\znth{\Z}{6}}}{W}\to
0
\]
where $W=W_n\subset \dual{\sq{F}}$. 

Since $\chi_0\in W$, 
\[
\upp{\ho{3}{\spl{2}{F}}{\znth{\Z}{6}}}{W}=
\upp{\left(\upp{\ho{3}{\spl{2}{F}}{\znth{\Z}{6}}}{\chi_0}\right)}{W}
=\upp{\left(\hot{F}{\znth{\Z}{6}}\right)}{W}.
\]

However, 
\[
\hot{F}{\znth{\Z}{6}}\cong \znth{\aug{F}\rrpbker{F}}{6}
\]
by Corollary \ref{cor:rpbker3}.

Thus 
\begin{eqnarray*}
\upp{\ho{3}{\spl{2}{F}}{\znth{\Z}{6}}}{W}&\cong& 
\upp{\left(\znth{\aug{F}\rrpbker{F}}{6}\right)}{W}\\
&\cong & \upp{\left(\znth{\rrpbker{F}}{6}\right)}{W}\\
&\cong & \indf{F_n}{F}{\left(\znth{\aug{F_n}\rrpbker{F_n}}{6}\right)}\\
&\cong & \indf{F_n}{F}{\left(\hot{F_n}{\znth{\Z}{6}}\right)}
 \mbox{ by Proposition \ref{prop:rpbwn}}.\\
\end{eqnarray*}

On the other hand, we have 
\[
\dwn{\ho{3}{\spl{2}{F}}{\znth{\Z}{6}}}{W}\cong
 \oplus_{\chi\in W}\dwn{\ho{3}{\spl{2}{F}}{\znth{\Z}{6}}}{\chi}.
\]

When $\chi=\chi_0$
\[
\dwn{\ho{3}{\spl{2}{F}}{\znth{\Z}{6}}}{\chi}=\znth{\kind{F}}{6}
\]
by Lemma \ref{lem:h3sl20} (2).

If $\chi\in W\setminus\{ \chi_0\}$ then
\[
\dwn{\ho{3}{\spl{2}{F}}{\znth{\Z}{6}}}{\chi}
=\dwn{\hot{F}{\znth{\Z}{6}}}{\chi}
\cong \dwn{\znth{\aug{F}\rrpbker{F}}{6}}{\chi}
\cong \dwn{\znth{\rrpbker{F}}{6}}{\chi}
\cong \znth{\redpb{F_{i(\chi)}}}{6}
\] 
by Proposition \ref{prop:rpbwn} again. 

Since $\znth{\redpb{F}}{6}=\znth{\pb{F}}{6}$ for any field $F$, and since for any 
$i\leq n$,  
\[
\card{\{ \chi\in W\ |\ i(\chi)=i\}}= \card{W_{i}\setminus W_{i-1}}= 2^{i-1}
\]
the result follows. 
\end{proof}

\begin{cor}\label{cor:rpbwn} 
Let $n\geq 1$ and let $F=F_0, F_1,\ldots, F_n$ be a sequence of fields such that 
for $i=1,\ldots, n$ there exists a discrete valuation $v_i$ on $F_{i-1}$ with residue field 
$F_i$. We suppose also that each of these valuations satisfies the condition that there exists 
$n_i\geq 1$ with $U_{n_i,v_i}\subset U_{1,v_i}^2$.

Suppose further that $\hot{F_n}{\znth{\Z}{6}}=0$ (eg. if $F_n$ is finite or algebraically closed 
or real closed). 

Then 
\[
\ho{3}{\spl{2}{F}}{\znth{\Z}{6}}\cong
\znth{\kind{F}}{6}\oplus \left(\bigoplus_{i=1}^{n}\znth{\pb{F_i}}{6}^{\oplus 2^{i-1}}\right)
\]

Furthermore, if $X^2-X+1$ has a root in $F$, we can 
replace $\frac{1}{6}$ with $\frac{1}{2}$ in this
statement.
\end{cor}

\begin{exa} Let $p$ be an odd prime number. Let $F=\laurs{\laurs{\F{p}}{x_n}\cdots}{x_1}$.

For any $n\geq 1$ we then have 
\[
\ho{3}{\spl{2}{F}}{\znth{\Z}{6}}\cong 
\znth{\kind{F}}{6}\oplus 
\left(\bigoplus_{i=1}^{n}\znth{\pb{\laurs{\laurs{\F{p}}{x_n}\cdots}{x_i}}}{6}^{\oplus 2^{i-1}}\right).
\] 

We will see in the next section that we can always 
replace $\frac{1}{6}$ with $\frac{1}{2}$ in this particular example. 
\end{exa}
\section{ $3$-torsion and the module $\cconstmod{F}$}\label{sec:3tors}

We begin by observing that we can refine the main result of section \ref{sec:main} as follows:

Let $F$ be a field with discrete valuation $v:F^\times \to \Z$. We consider the module
\[
\cconstmod{F}(v):= \ker{\spec{v}: \aug{F}\cconstmod{F}\to \aug{F}(\indf{k}{F}{\cconstmod{k}})}.
\]
Observe that this module is annihilated by $3$.

\begin{prop}\label{prop:mainsv3}
Let $F$ be a field with discrete valuation $v:F^\times \to \Z$. Suppose that there exists some 
$n\geq 1$ such that $U_n\subset U_1^2$. Then the homomorphism 
\[
\spec{v}:\zhalf{\rpbker{F}}\to \indf{k}{F}{\zhalf{\rpbker{k}}}
\]
induces a surjective homomorphism 
\[
\aug{F}\zhalf{\rpbker{F}}\to\aug{F}\left(\indf{k}{F}{\zhalf{\rpbker{k}}}\right)
\]
whose kernel is $\cconstmod{F}(v)$. 
\end{prop}

\begin{proof}
There is a natural commutative diagram of $\zhalf{\sgr{F}}$-modules with exact rows
\begin{eqnarray}\label{comm}
\xymatrix{
0\ar[r]
&\cconstmod{F}\ar[r]\ar@{>>}^-{\spec{v}}[d]
&\zhalf{\rpbker{F}}\ar[r]\ar@{>>}^-{\spec{v}}[d]
&\zhalf{\rrpbker{F}}\ar[r]\ar@{>>}^-{\spec{v}}[d]
&0\\
0\ar[r]
&\indf{k}{F}{\cconstmod{k}}\ar[r]
&\indf{k}{F}{\zhalf{\rpbker{k}}}\ar[r]
&\indf{k}{F}{\zhalf{\rrpbker{k}}}\ar[r]
&0\\}
\end{eqnarray}

By Corollary \ref{cor:exact1}, the diagram remains exact after applying the functor 
$M\to \aug{F}M$. By Theorem \ref{thm:mainsv}, the right-hand vertical arrow then 
becomes an isomorphism. The statement then follows from the snake lemma.
\end{proof}

From this we deduce, as in section \ref{sec:main}:

\begin{cor}\label{cor:ifrpf2}
Let $F$ be a field with discrete valuation $v:F^\times \to \Z$  and residue field $k$.
  Suppose that there exists some 
$n\geq 1$ such that $U_n\subset U_1^2$. Suppose further that $\cconstmod{F}(v)=0$.

Let 
$W=W_v:=\{ \chi_0,\chi_v\}$. 
Let $\pi\in F^\times$ satisfy 
$v(\pi)=1$. 

Then there are natural isomorphisms of $\sgr{F}$-modules
\begin{eqnarray*}
\upp{\zhalf{\rpbker{F}}}{W}&\cong &\indf{k}{F}{\left(\aug{k}\zhalf{\rpbker{k}}\right)}\\
\end{eqnarray*}
and
\[
\dwn{\zhalf{\rpbker{F}}}{\chi_v}\cong \pieps{\left(\zhalf{\pb{k}}\right)}{\pi}{-1}.
\]
\end{cor}

Since 
\[
\aug{F}\zhalf{\rpbker{F}} \cong \hot{F}{\zhalf{\Z}}=\aug{F}\ho{3}{\spl{2}{F}}{\zhalf{\Z}}
\]
we deduce:

\begin{cor}\label{cor:hot2}
Let $F$ be a field with discrete valuation $v:F^\times \to \Z$  and residue field $k$.
  Suppose that there exists some 
$n\geq 1$ such that $U_n\subset U_1^2$. Suppose further that $\cconstmod{F}(v)=0$.

Let 
$W=W_v:=\{ \chi_0,\chi_v\}$. 
Let $\pi\in F^\times$ satisfy 
$v(\pi)=1$. 

Then there are natural isomorphisms of $\sgr{F}$-modules
\begin{eqnarray*}
\upp{\ho{3}{\spl{2}{F}}{\zhalf{\Z}}}{W}&\cong &\indf{k}{F}{\left(\hot{k}{\zhalf{\Z}}\right)}\\
\end{eqnarray*}
and
\[
\dwn{\ho{3}{\spl{2}{F}}{\zhalf{\Z}}}{W}\cong \zhalf{\kind{F}}\oplus\zhalf{\pb{k}}.
\]


\end{cor}

Using this result together with induction on $n$, we deduce as in \ref{thm:rpbwn}:

\begin{cor}\label{cor:hofn2} 
Let $n\geq 1$ and let $F=F_0, F_1,\ldots, F_n$ be a sequence of fields such that 
for $i=1,\ldots, n$ there exists a discrete valuation $v_i$ on $F_{i-1}$ with residue field 
$F_i$.

 Suppose furthermore that 
\begin{enumerate}
\item    For each $i$ there exists $n_i$ such that $U_{v_i,n_i}\subset U_{v_i,1}^2$

\item $\cconstmod{F_{i-1}}(v_i)=0$ for $1\leq i \leq n$
\end{enumerate}
Then there is a natural short exact sequence 
\[
0\to 
\indf{F_n}{F}{\left(\hot{F_n}{\znth{\Z}{2}}\right)}\to 
\ho{3}{\spl{2}{F}}{\znth{\Z}{2}}\to
\znth{\kind{F}}{2}\oplus \left(\bigoplus_{i=1}^{n}\znth{\pb{F_i}}{2}^{\oplus 2^{i-1}}\right)
\to 0. 
\]
In particular, if $F_n$ is finite or real-closed or algebraically closed then
\[
\ho{3}{\spl{2}{F}}{\znth{\Z}{2}}\cong
\znth{\kind{F}}{2}\oplus \left(\bigoplus_{i=1}^{n}\znth{\pb{F_i}}{2}^{\oplus 2^{i-1}}\right).
\]
\end{cor}

In the remainder of this section, we describe some circumstances under which this kernel,
$\cconstmod{F}(v)$, is trivial. Note that a sufficient condition for this to happen is that 
the surjective homomorphism $\spec{v}:\cconstmod{F}\to \indf{k}{F}{\cconstmod{k}}$ be an isomorphism.

Recall from section \ref{sec:cconst} that $E$ is the field obtained from $F$ by 
adjoining a root of $X^2-X+1$, $\tnorm{F}=\pm N_{E/F}(E^\times)$ , $\rsq{F}=F^\times/\tnorm{F}$ 
and $\rsgr{F}$ is the group algebra $\F{3}[\rsq{F}]$.

\begin{lem}\label{lem:dfv}
 Let $F$ be a field with discrete valuation $v$ with residue field 
$k$. Then $\cconstmod{F}(v)=0$ 
under any of the following circumstances:
\begin{enumerate} 
\item The polynomial $X^2-X+1$ has a root in $F$.
\item $v$ is complete and $k$ is finite of characteristic not equal to $3$
\item $\chara{F}=0$ and $v$ is complete and $k$ is finite of characteristic $3$ and either 
$[F:\Q_3]$ is odd or $\sqrt{-3}\in F$.
\end{enumerate}
\end{lem}

\begin{proof}\ 

\begin{enumerate}
\item This is clear since in this case $\cconstmod{F}=\cconstmod{k}=0$.
\item
Suppose that $X^2-X+1$ has no root in $F$. Since $\chara{k}\not=3$, it follows that
 $X^2-X+1$ has no root in $k$. 

Thus $\cconstmod{F}\not=0$ and $\cconstmod{k}\not= 0$. 
Since $k$ is a finite field, $k^\times =\tnorm{k}$, $\rsgr{k}=\F{3}$  and  hence 
$\cconstmod{k}=\F{3}\cconst{k}$ is a $1$-dimensional vector space over $\F{3}$. Thus
$\indf{k}{F}{\cconstmod{k}}=\cconstmod{k}\oplus \cconstmod{k}$ is $2$-dimensional. 

Now since $E/F$ is a quadratic extension of local fields, we have 
\[
\card{F^\times/N_{E/F}(E^\times)}=[E:F]=2.
\]
However, the homomorphism $\spec{v}: \cconstmod{F}\to \indf{k}{F}{\cconstmod{k}}$ 
is surjective and hence
\[
\dim{\F{3}}{\cconstmod{F}}\leq \dim{\F{3}}{\rsgr{F}}\leq 
\dim{\F{3}}{\F{3}[F^\times/N_{E/F}(E^\times)]}
= 2
\]
so that $\spec{v}$ is an isomorphism as required.
\item  This particular case is proved in \cite[Corollary 6.3]{hut:rbl11}.
\end{enumerate}
\end{proof}
\begin{rem} Of course, the case where $\chara{F}=3=\chara{k}$ is included in case (1) of this 
lemma.
\end{rem}

Combining Lemma \ref{lem:dfv} with Corollary \ref{cor:hot2},  we extend \cite[Theorem 6.19]{hut:rbl11} 
to include $2$-adic local fields:

\begin{prop}\label{prop:qp}
Let $F$ be a field complete with respect to a discrete valuation $v$ and finite residue field $k$. 

If $\chara{k}=2$, we suppose that $\Q_2\subset F$. If $\Q_3\subset F$, suppose that $[F:\Q_3]$ is odd. 

Then 
\[
\ho{3}{\spl{2}{F}}{\zhalf{\Z}}\cong \zhalf{\pb{k}}\oplus \zhalf{\kind{F}}.
\]
\end{prop}

\begin{rem}
The condition on fields of residue characteristic $2$ is included to ensure that the 
hypothesis $U_n\subset U_1^2$ holds.
The condition on fields of residue characteristic $3$ is included to ensure that 
the hypothesis $\cconstmod{F}(v)=0$ holds.
\end{rem}

\begin{exa}
In particular we have
\[
\ho{3}{\spl{2}{\Q_p}}{\zhalf{\Z}}\cong \zhalf{\pb{\F{p}}}\oplus \zhalf{\kind{\Q_p}}.
\]
for \emph{all} primes $p$. 
\end{exa}

\begin{lem}\label{lem:rsq}
Let $F$ be a field with discrete valuation $v$ and residue field $k$. 
Suppose that $U_1=U_1^2$, that the polynomial $X^2-X+1$ has no root in  $k$. 

Then the $\sgr{F}$-algebra  homomorphism 
\[
\sgr{F}\to \indf{k}{F}{(\sgr{k})}
=\sgr{F}\otimes_{\Z[U/U^2]}\sgr{k}, \quad \an{a}\mapsto \an{a}\otimes 1
\] 
induces an isomorphism $\rsgr{F}\cong \indf{k}{F}{(\rsgr{k})}$.
\end{lem}
\begin{proof} 
Since $U_1=U_1^2$, we have $U/U^2\cong \sq{k}$ and hence $\Z[U/U^2]=\sgr{k}$. Thus 
$\indf{k}{F}{(\sgr{k})}=\sgr{F}\otimes_{\sgr{k}}\sgr{k}=\sgr{F}$ canonically.

Let $\pi\in F$ be a uniformizer. Let $C_\pi$ denote the multiplicative group of order $2$ 
generated by the square class of $\pi$. Then 
\[
\sq{F}\cong U/U^2\times C_\pi \cong \sq{k}\times C_\pi 
\]
and hence $\sgr{F}=\sgr{k}[C_\pi]$ and $(\rsgr{k})_F=\rsgr{k}[C_\pi]$.

Thus, it will be enough to show that 
\[
F^\times/\tnorm{F}=(k^\times/\tnorm{k})\times C_\pi.
\]

Let $\tilde{\mathcal{N}}_k=\{ u\in U\ |\ \bar{u}\in\tnorm{k}\}$. 
Thus we require to show that 
\[
\tnorm{F}=\pi^{2\Z}\cdot \tilde{\mathcal{N}}_k. 
\]

Let $u\in \tilde{\mathcal{N}}_k$. Thus $\pm\bar{u}=\bar{s}^2-\bar{s}\bar{t}+\bar{t}^2$ for 
some $s,t\in U$. 
Thus $\pm u=(s^2-st+t^2)w^2$ for some $w\in U_1=U_1^2$, and hence $u\in \tnorm{F}$.

It remains to show that $\pi\not\in \tnorm{F}$: Suppose that $a,b\in F^\times$. Then 
$v(a^2-ab+b^2)\equiv v(1-c+c^2)\pmod{2}$ where $c=b/a$. If $v(c)\not=0$ then $v(1-c+c^2)\equiv 0\pmod{2}$. On the other 
hand if $v(c)=0$ then $v(1-c+c^2)=0$ since $1-\bar{c}+\bar{c}^2\not= 0$ in $k$ by hypothesis. 

It follows that if $x\in\tnorm{F}$ then $v(x)$ is even, and thus $\pi\not\in\tnorm{F}$.    
\end{proof}

\begin{cor}\label{cor:dffree}
Let $F$ be a field with discrete valuation $v$ and residue field $k$. 
Suppose that $U_1=U_1^2$, that the polynomial $X^2-X+1$ has no root in  $k$.

Suppose further that $\cconstmod{k}$ is a free 
rank one $\rsgr{k}$-module. 

Then  $\cconstmod{F}$ is a free 
rank one $\rsgr{F}$-module and the homomorphism 
\[
\spec{v}: \cconstmod{F}\to \indf{k}{F}{(\cconstmod{k})}
\]
is an isomorphism and $\cconstmod{F}(v)=0$.
\end{cor}
\begin{proof}
The diagram 
\[
\xymatrix{
\rsgr{F}\ar@{->>}[d]\ar^-{\cong}[r]
&
\indf{k}{F}{(\rsgr{k})}\ar^{\cong}[d]\\
\cconstmod{F}\ar@{->>}^-{\spec{v}}[r]
&
\indf{k}{F}{(\cconstmod{k})}\\
}
\]
commutes, since $\spec{v}(\an{a}\cconst{F})=\an{a}\otimes \cconst{k}$.
\end{proof}

In view of the hypothesis of Corollary \ref{cor:dffree}, we itemize some examples of 
fields $F$ for which $\cconstmod{F}$ is a free $\rsgr{F}$-module. 

\begin{prop}\label{prop:dffree}
Let $F$ be a field in which $X^2-X+1$ has no roots. Then $\cconstmod{F}$ is a free 
$\rsgr{F}$-module of rank $1$ in the following cases:
\begin{enumerate}
\item $F$ is a finite field.
\item $F$ is a real closed field.
\item $F$ is complete with respect to a discrete valuation $v$ with finite residue field of 
characteristic not equal to $3$.
\item $\chara{F}=0$ and $F$ is complete with respect to a discrete valuation 
$v$ with finite residue field of characteristic not $3$ and $[F:\Q_3]$ is odd. 
\end{enumerate}
\end{prop}

\begin{proof}\ 

\begin{enumerate}
\item As observed in the proof of  Lemma \ref{lem:dfv} (2), $\tnorm{F}=F^\times$ in this 
case and thus $\rsgr{F}=\F{3}$. Thus $\cconstmod{F}=\F{3}\cdot\cconst{F}$ is free of 
rank $1$ as a $\rsgr{F}$-module.
\item In this case also we must have $\tnorm{F}=F^\times$.
\item When $\chara{k}\not=3$ the argument in the proof of Lemma \ref{lem:dfv} (2), shows that 
\[
\dim{\F{3}}{\cconstmod{F}}=2=\dim{\F{3}}{\rsgr{F}}
\]
so that the conclusion follows at once.
\item In the proof of \cite[Lemma 6.2]{hut:rbl11} it is shown that $\tnorm{F}=F^\times$ under 
the given hypotheses.  
\end{enumerate}
\end{proof}

Combining Corollary \ref{cor:dffree} and Proposition \ref{prop:dffree} with Corollary 
\ref{cor:hofn2} deduce:

\begin{cor}\label{cor:fianl} 
Let $n\geq 1$ and let $F=F_0, F_1,\ldots, F_n$ be a sequence of fields such that 
for $i=1,\ldots, n$ there exists a complete discrete valuation $v_i$ on $F_{i-1}$ with residue field 
$F_i$.

Suppose that $F_n$ is real-closed or finite or that $F_n$ is quadratically closed of characteristic 
not $2$. Suppose furthermore that 
\begin{enumerate}
\item   $\chara{F_{n-1}}\not=2$ \emph{or}
\item $\chara{F_n}\not=3$ \emph{or}
\item $\chara{F_n}=3$ and  $\chara{F_{n-1}}=0$ and either $[F_{n-1}:\Q_3]$ is odd or 
$\sqrt{-3}\in F_{n-1}$.
\end{enumerate}
Then 
\[
\ho{3}{\spl{2}{F}}{\zhalf{\Z}}\cong \zhalf{\kind{F}}\oplus
\left( \bigoplus_{i=1}^{n}\zhalf{\pb{F_i}}^{\oplus 2^{i-1}}\right).
\]
\end{cor}

\begin{exa}
Let $p$ be an odd prime number. Let $F=\laurs{\laurs{\F{p}}{x_n}\cdots}{x_1}$.

For any $n\geq 1$ we then have 
\[
\ho{3}{\spl{2}{F}}{\znth{\Z}{2}}\cong 
\znth{\kind{F}}{2}\oplus 
\left(\bigoplus_{i=1}^{n}\znth{\pb{\laurs{\laurs{\F{p}}{x_n}\cdots}{x_i}}}{2}^{\oplus 2^{i-1}}\right).
\] 
\end{exa}

\begin{exa}\label{exa:qpxx} Let $p$ be any prime, let $n\geq 2$ and  and let 
$F=\laurs{\laurs{\Q_p}{x_{n-1}}\cdots}{x_1}$. Then 
 \[
\ho{3}{\spl{2}{F}}{\znth{\Z}{2}}\cong 
\znth{\kind{F}}{2}\oplus 
\left(\bigoplus_{i=1}^{n}\znth{\pb{F_i}}{2}^{\oplus 2^{i-1}}\right)
\]
where $F_i= \laurs{\laurs{\Q_p}{x_{n-1}}\cdots}{x_{i+1}}$ for $i\leq n-2$, $F_{n-1}=\Q_p$  and $F_n=\F{p}$. 
\end{exa}

\begin{exa}\label{exa:rxx} Let 
$F=\laurs{\laurs{\R}{x_{n}}\cdots}{x_1}$. Then 
 \[
\ho{3}{\spl{2}{F}}{\znth{\Z}{2}}\cong 
\znth{\kind{F}}{2}\oplus 
\left(\bigoplus_{i=1}^{n}\znth{\pb{F_i}}{2}^{\oplus 2^{i-1}}\right)
\]
where $F_i= \laurs{\laurs{\R}{x_{n}}\cdots}{x_{i+1}}$ for $i\leq n-1$ and $F_{n}=\R$. 
\end{exa}

\appendix
\section{The specialization map $\spec{v}$ revisited}\label{sec:appendix}
In this appendix we revisit the specialization map $\spec{v}$ introduced in 
\cite{hut:rbl11} and show that it can be defined with an improved target. 
This improved version of the specialization homomorphism is required for our 
calculations in section \ref{sec:3tors} above.

Given a field $F$ with a valuation $v:F^\times\to \Gamma$ and corresponding residue field $k$, we 
let $U=U_v$ denote the 
corresponding group of units and $U_1=U_{1,v}$ the 
units mapping to $1$ in $k^\times$. The natural quotient map  
$\mathcal{O}_v\to k$ will be denoted $a\mapsto \bar{a}$. 

Given any $\sgr{k}$-module $M$, 
we define the induced $\sgr{F}$-module
\[
\indf{k}{F}{M}:=\sgr{F}\otimes_{\Z[U]}M=\sgr{F}\otimes_{\Z[U/U^2]}M.
\]

In \cite{hut:rbl11} we proved the existence of a natural surjective specialization homomorphism 
\[
\spec{v}:\rpb{F}\to \indf{k}{F}{\rrrpb{k}}
\]
where $\rrrpb{k}:= \rpb{k}/(\ks{1}{k}+\aug{k}\cconstmod{k})=\rpb{k}/(\ks{1}{k}+\ks{2}{k})$.

In this appendix, we revisit this specialization map and prove that it can 
be defined with target $\indf{k}{F}{\qrpb{k}}$, rather than $\rrrpb{k}_F$.

\begin{lem}\label{lem:cfsus1}
For any field $F$ we have $3\bconst{F}=\suss{1}{-1}$.
\end{lem}

\begin{proof}
Let $a\in F^\times\setminus\{ 1\}$. Then
\begin{eqnarray*}
3\bconst{F}&=& C(a)+C(1-a^{-1})+C\left( \frac{1}{1-a}\right)\\
&=& \gpb{a}+\an{-1}\gpb{1-a}+\pf{\frac{1}{1-a}}\suss{1}{a}\\
&+& \gpb{1-a^{-1}}+\an{-1}\gpb{a^{-1}}+\pf{a}\suss{1}{1-a^{-1}}\\
&+& \gpb{\frac{1}{1-a}}+\an{-1}\gpb{\frac{1}{1-a^{-1}}}+\pf{1-a^{-1}}\suss{1}{\frac{1}{1-a}}\\
&=& \suss{1}{a}+\suss{1}{1-a^{-1}}+\suss{1}{\frac{1}{1-a}}\\
&+& \pf{\frac{1}{1-a}}\suss{1}{a}+\pf{a}\suss{1}{1-a^{-1}}+\pf{1-a^{-1}}\suss{1}{\frac{1}{1-a}}\\
&=& \an{\frac{1}{1-a}}\suss{1}{a}+\an{a}\suss{1}{1-a^{-1}}+\an{1-a^{-1}}\suss{1}{\frac{1}{1-a}}\\
&=& \suss{1}{\frac{1}{a^{-1}-1}}-\suss{1}{\frac{1}{1-a}}+\suss{1}{a-1}-\suss{1}{a}
+\suss{1}{-\frac{1}{a}}-
\suss{1}{1-a^{-1}}\\
&=& 3\suss{1}{-1}=\suss{1}{-1}
\end{eqnarray*}
since for any $b$, we have 
\[
\suss{1}{b}-\suss{1}{-b^{-1}}= \suss{1}{b}-\an{-1}\suss{1}{b^-1}+\suss{1}{-1}=\suss{1}{-1}
\]
by Lemma  \ref{lem:sus1} (1).
\end{proof}

\begin{lem}\label{lem:df}  
For all $a\in F^\times$
\[
\pf{a}\bconst{F}=\suss{2}{a}-\suss{1}{a^{-1}}. 
\]
\end{lem}

\begin{proof}
Since $3\bconst{F}=\suss{1}{-1}$ by Lemma \ref{lem:cfsus1} and $2\bconst{F}=\cconst{F}$, 
we have $\bconst{F}=\suss{1}{-1}-\cconst{F}$. 

Thus, if $a\in F^\times$ we have 
\[
\pf{a}\bconst{F}=\pf{a}\suss{1}{-1}-\pf{a}\cconst{F}=\pf{a}\suss{1}{-1}-\pf{-a}\cconst{F}
\]
since $\pf{-1}\cconst{A}=0$. Thus, by Theorem \ref{thm:df}
\begin{eqnarray*}
\pf{a}\bconst{F}&=&\pf{a}\suss{1}{-1}-\suss{1}{-a}+\suss{2}{-a}\\
&=& \an{a}\suss{1}{-1}-\suss{1}{-1}-\an{a}\suss{1}{-1}-\suss{1}{a}+\suss{2}{-a}\\
&=& -(\suss{1}{-1}+\suss{1}{a})+\suss{2}{-a}\\
&=& -(\suss{1}{-1}+\an{-1}\suss{1}{a^{-1}})+\suss{2}{-a}\\
&=& \suss{2}{-a}-\suss{1}{-a^{-1}}.\\
\end{eqnarray*}

However, $\pf{-1}\bconst{F}=0$. 
Thus $\pf{a}\bconst{F}=\pf{-a}\bconst{F}=\suss{2}{a}-\suss{1}{a^{-1}}$.
\end{proof}

\begin{cor}\label{cor:bconst}
Let $F$ be a field and let $a\in F^\times$. 
\begin{enumerate}
\item
 $\an{a}\bconst{F}=\bconst{F}+\suss{2}{a}$ in $\qrpb{F}$.
\item
$\suss{2}{a}=\suss{2}{a^{-1}}=\suss{2}{-a}$ in $\qrpb{F}$.
\end{enumerate}
\end{cor}
\begin{proof} 
\begin{enumerate}
\item
This is an immediate consequence of the formula in Corollary \ref{lem:df} above.
\item
This follows immediately from (1) since $\an{a}=\an{a^{-1}}$ and $\an{-1}\bconst{F}=\bconst{F}$.
\end{enumerate}
\end{proof}

Note also, that in $\qrpb{F}$, we have, by definition, that $\bconst{F}=\gpb{x}+\an{-1}\gpb{1-x}$ 
for any $x\not=
0,1$.

\begin{thm}\label{thm:spec}
There is a surjective $\sgr{F}$-module homomorphism   
\begin{eqnarray*}
S_{v}:\rpb{F}&\to&\indf{k}{F}{\qrpb{k}}\\
\gpb{a}&\mapsto&\left\{
\begin{array}{rc}
1\otimes \gpb{\bar{a}},& v(a)=0\\
1\otimes\bconst{k},& v(a)>0\\
-(1\otimes\bconst{k}),& v(a)<0\\
\end{array}
\right.
\end{eqnarray*}
\end{thm}

\begin{rem} The proof we give here follows closely the proof of Theorem 4.9 in \cite{hut:rbl11}. Only the following 
cases differ: Case (v) (a) and (b), Case (vii). These cases use Corollary \ref{cor:bconst}, which in turn relies on 
Theorem \ref{thm:df}. 
\end{rem}

\begin{proof}
Let $Z_1$ denote the set of symbols of the form $\gpb{x}, x\not= 1$ and let 
$T:\sgr{F}[Z_1]\to \indf{k}{F}{\qrpb{k}}$ be the unique $\sgr{F}$-homomorphism given by  
\begin{eqnarray*}
\gpb{a}&\mapsto&\left\{
\begin{array}{rc}
1\otimes \gpb{\bar{a}},& v(a)=0\\
1\otimes\bconst{k},& v(a)>0\\
-(1\otimes\bconst{k}),& v(a)<0\\
\end{array}
\right.
\end{eqnarray*}
We must prove that $T(S_{x,y})=0$ for all $x,y\in F^\times\setminus\{ 1\}$.

Through the remainder of this proof we will adopt the following notation: Given $x,y\in F^\times\setminus\{ 1\}$, 
we let
\[
u=\frac{y}{x}\mbox{ and } w=\frac{1-x}{1-y}.
\]
Note that 
\[
\frac{1-x^{-1}}{1-y^{-1}}=\frac{y}{x}\cdot\frac{x-1}{y-1}=uw.
\]
Thus, with this notation, $S_{x,y}$ becomes $\gpb{x}-\gpb{y}+\an{x}\gpb{u}-\an{x^{-1}-1}\gpb{uw}+\an{1-x}\gpb{w}$.

We divide the proof into several cases:
\begin{enumerate}
\item[Case (i):] $v(x),v(y)\not= 0$
\begin{enumerate}
\item[Subcase (a):] $v(x)=v(y)>0$. 

Then $1-x,1-y\in U_1$ and hence $w\in U_1$, so that $\bar{w}=1$ and $\overline{uw}=\bar{u}$. Thus 
\[
T(S_{x,y})=1\otimes\bconst{k}-1\otimes\bconst{k}+\an{x}\otimes\gpb{\bar{u}}-\an{x^{-1}-1}\otimes\gpb{\bar{u}}.
\]
However, $x^{-1}-1=x^{-1}(1-x)$, so that 
$\an{x^{-1}-1}\otimes\bar{u}=\an{x}\otimes\an{1-\bar{x}}\gpb{\bar{u}}=\an{x}\otimes\gpb{\bar{u}}$, and thus 
$T(S_{x,y})=0$ as required.

\item[Subcase (b):] $v(x)=v(y)<0$.

Then $u\in U$ and $uw\in U_1$ so that $\bar{w}=\bar{u}^{-1}$. Thus
\[
T(S_{x,y})=-1\otimes\bconst{k}+1\otimes\bconst{k}+\an{x}\otimes\gpb{\bar{u}}+\an{1-x}\otimes\gpb{\bar{u}^{-1}}.
\]

But $1-x=-x(1-x^{-1})$ and $1-x^{-1}\in U_1$, so that the last term is $\an{-x}\otimes\gpb{\bar{u}^{-1}}$ and 
hence $T(S_{x,y})=\an{x}\otimes\suss{1}{\bar{u}}=0$ in $\indf{k}{F}{\qrpb{k}}$.


\item[Subcase (c):] $v(x)>v(y)>0$.

Then $w\in U_1$ and $v(u), v(uw)<0$. So
\[
T(S_{x,y})=1\otimes \bconst{k}-1\otimes\bconst{k}-\an{x}\otimes\bconst{k}+\an{x^{-1}-1}\otimes\bconst{k}.
\]
But since $x^{-1}-1=x^{-1}(1-x)$ and $1-x\in U_1$ it follows that $\an{x^{-1}-1}\otimes\bconst{k}=\an{x}\otimes\bconst{k}$
and hence $T(S_{x,y})=0$.
\item[Subcase (d):] $v(x)>0>v(y)$.

Then $v(u)=v(y)-v(x)<0$, $v(w)=-v(1-y)=-v(y(y^{-1}-1))=-v(y)>0$ and $v(uw)=v(u)+v(w)=-v(x)<0$. So 
\[
T(S_{x,y})=1\otimes\bconst{k}+1\otimes\bconst{k}-\an{x}\otimes \bconst{k}+\an{x^{-1}-1}\otimes\bconst{k}+
\an{1-x}\otimes\bconst{k}.
\]
But $1-x\in U_1$ and $\an{x^{-1}-1}=\an{x}\an{1-x}$. So this gives $T(S_{x,y})=1\otimes 3\bconst{k}$. However,
$3\bconst{k}=\suss{1}{-1}$ in $\qrpb{k}$.
\item[Subcase (e):] $0>v(x)>v(y)$.

Then $v(u)=v(y)-v(x)<0$ and $uw\in U_1$, so that $v(w)=-v(u)>0$ and $\bar{w}=\bar{u}^{-1}$. Thus 
\[
T(S_{x,y})=-(1\otimes\bconst{k})+1\otimes\bconst{k}-\an{x}\otimes\bconst{k}+\an{1-x}\otimes\bconst{k}.
\]
Now $1-x=-x(1-x^{-1})$ and $1-x^{-1}\in U_1$, so the last term is 
$\an{x}\otimes\an{-1}\bconst{k}=\an{x}\otimes\bconst{k}$. This gives 
$T(S_{x,y})=0$ as required.
\item[Subcases (f),(g),(h):] The corresponding calculations when $v(y)>v(x)$ are almost identical.
\end{enumerate}
\item[Case (ii):] $x,y\in U_1$.
\begin{enumerate}
\item[Subcase (a):] $v(1-x)\not= v(1-y)$.

Then $u\in U_1$ and $v(w)=v(uw)\not= 0$. So 
\[
T(S_{x,y})=\pm\left( \an{x^{-1}-1}\otimes \bconst{k}-\an{1-x}\otimes\bconst{k}\right)=0
\]
since $\an{x^{-1}-1}=\an{x^{-1}}\an{1-x}$ and $x^{-1}\in U_1$.

\item[Subcase (b):] $v(1-x)=v(1-y)$.

Then $u\in U_1$ and $w,uw\in U$ with $\bar{uw}=\bar{w}$. So
\[
T(S_{x,y})=-\an{x^{-1}-1}\otimes\gpb{\bar{w}}+\an{1-x}\otimes\gpb{\bar{w}}
\]
which is $0$ by the same argument as the previous (sub)case.
\end{enumerate}

\item[Case (iii):] $x\in U_1$, $v(y)\not=0$. 

Then $v(u)=v(y)$. Observe that $v(1-y)=\mathrm{min}(v(1),v(y))=\mathrm{min}(0,v(y))\leq 0$. Of course, $v(1-x)>0$.

Thus $v(w)=v(1-x)-v(1-y)>0$ and $v(uw)=v(u)+v(w)=v(1-x)+v(y)-v(1-y)>0$ 
since $v(y)-v(1-y)\geq 0$. So 
\[
T(S_{x,y})=\pm(1\otimes\bconst{k}-1\otimes\bconst{k})-\an{x^{-1}-1}\otimes \bconst{k}+\an{1-x}\otimes\bconst{k}=0
\]
since $x\in U_1$ and thus $\an{x^{-1}-1}\otimes\bconst{k}\an{1-x}\otimes\an{\bar{x}^{-1}}\bconst{k}
=\an{1-x}\otimes\bconst{k}$.

\item[Case (iv):] $v(x)\not=0$, $y\in U_1$

Arguing as in the last case, $v(w),v(uw)<0$ in this case.

\begin{enumerate}
\item[Subcase (a):] $v(x)>0$

Then $v(u)=-v(x)<0$. So 
\[
T(S_{x,y})=1\otimes \bconst{k}-\an{x}\otimes\bconst{k}+\an{x^{-1}-1}\otimes\bconst{k}-\an{1-x}\otimes\bconst{k}=0
\]
since $1-x\in U_1$.

\item[Subcase (b):] $v(x)<0$.

Then $v(u)>0$ and
\[
T(S_{x,y})=-1\otimes \bconst{k}+\an{x}\otimes\bconst{k}+\an{x^{-1}-1}\otimes\bconst{k}-\an{1-x}\otimes\bconst{k}
\]
since $1-x^{-1}\in U_1$.

\end{enumerate}
\item[Case (v):] $x\in U\setminus U_1$ and $v(y)\not=0$.

\begin{enumerate}
\item[Subcase (a):] $v(y)>0$.

Then $v(u)=v(y)>0$. Since $1-x\in U$ and  $1-y\in U_1$, it follows that $v(w)=0$, $\bar{w}=1-\bar{x}$ in $k$ and 
$v(uw)>0$. Thus 
\begin{eqnarray*}
T(S_{x,y})&=&1\otimes\gpb{\bar{x}}-1\otimes\bconst{k}+\an{x}\otimes\bconst{k}-\an{x^{-1}-1}\otimes\bconst{k}
+\an{1-x}\otimes\gpb{1-\bar{x}}=1\otimes X.\\
\end{eqnarray*}
Using Corollary \ref{cor:bconst} (1) we have 
\begin{eqnarray*}
X= \gpb{\bar{x}}+\an{1-\bar{x}}\gpb{1-\bar{x}}-\bconst{k}+\suss{2}{\bar{x}}-\suss{2}{\bar{x}^{-1}-1}\in \qrpb{k}.
\end{eqnarray*}
But
\[
\suss{2}{\bar{x}^{-1}-1}=\suss{2}{\bar{x}^{-1}\cdot (1-\bar{x})}=\an{\bar{x}}\suss{2}{1-\bar{x}}+\suss{2}{\bar{x}^{-1}}
\]
and thus 
\[
X= \gpb{\bar{x}}+\an{1-\bar{x}}\gpb{1-\bar{x}}-\bconst{k} -\an{\bar{x}}\suss{2}{1-\bar{x}}
\]
(using Corollary \ref{cor:bconst} (2)). However,
\[
\bconst{k}=\gpb{\bar{x}}+\an{-1}\gpb{1-\bar{x}}\mbox{ and } 
\an{\bar{x}}\suss{2}{1-\bar{x}}=\an{1-\bar{x}}\gpb{1-\bar{x}}+\gpb{\frac{1}{1-\bar{x}}}.
\]
Thus
\[
X= -\an{-1}\gpb{1-\bar{x}}-\gpb{\frac{1}{1-\bar{x}}}=-\suss{1}{\frac{1}{1-\bar{x}}}=0 \mbox{ in }\qrpb{k}.
\]
\item[Subcase (b):] $v(y)<0$. 

We have $v(u)=v(y)<0$ and $v(w)=-v(1-y)=-v(y)>0$. Thus $v(uw)=v(u)+v(w)=v(y)-v(y)=0$. Furthermore, since 
$1-y^{-1}\in U_1$, $\bar{uw}=1-\bar{x}^{-1}$.

Thus $T(S_{x,y})=1\otimes X$ where 
\[
X= \gpb{\bar{x}}+\bconst{k}-\an{\bar{x}}\bconst{k}-\an{\bar{x}^{-1}-1}\gpb{1-\bar{x}^{-1}}+\an{1-\bar{x}}\bconst{k}\in 
\qrpb{k}.
\]
Using Corollary \ref{cor:bconst} (1) again, and the identity $\bconst{k}= \gpb{1-\bar{x}^{-1}}+\an{-1}\gpb{\bar{x}^{-1}}$ 
we deduce 
\begin{eqnarray*}
x&=& \suss{1}{\bar{x}}+\gpb{1-\bar{x}}-\an{\bar{x}^{-1}-1}\gpb{1-\bar{x}^{-1}}+\suss{2}{1-\bar{x}}-\suss{2}{\bar{x}}\\
&=& \gpb{1-\bar{x}}-\an{\bar{x}^{-1}-1}\gpb{1-\bar{x}^{-1}}+\suss{2}{1-\bar{x}}-\suss{2}{\bar{x}}\\
 &=&\gpb{1-\bar{x}}+\an{1-\bar{x}^{-1}}\gpb{\frac{1}{1-\bar{x}^{-1}}}+\suss{2}{1-\bar{x}}-\suss{2}{\bar{x}}\\
\end{eqnarray*}
(using $\suss{1}{1-\bar{x}^{-1}}=0$ in $\qrpb{k}$ in the last step).

However
\begin{eqnarray*}
\gpb{1-\bar{x}}+\an{1-\bar{x}^{-1}}\gpb{\frac{1}{1-\bar{x}^{-1}}}&=& \an{1-\bar{x}}\suss{2}{\frac{1}{1-\bar{x}^{-1}}}\\
&=& \suss{2}{\frac{1-\bar{x}}{1-\bar{x}^{-1}}}-\suss{2}{1-\bar{x}}\\
&=& \suss{2}{-\bar{x}}-\suss{2}{1-\bar{x}}
\end{eqnarray*}
and so $X=0$ by Corollary \ref{cor:bconst} (2) again.
\end{enumerate}
\item[Case (vi):] $y\in U\setminus U_1$ and $v(x)\not= 0$
\begin{enumerate}
\item[Subcase (a):] $v(x)>0$

we have $v(u)=-v(x)<0$, $v(w)=0$ and $1-x\in U_1$ so that $\bar{w}=(1-\bar{y})^{-1}$. Finally $v(uw)=v(u)<0$. Thus
\begin{eqnarray*}
T(S_{x,y})&=&1\otimes\bconst{k}-1\otimes\gpb{\bar{y}}-\an{x}\otimes\bconst{k}+\an{x^{-1}-1}\otimes\bconst{k}
+\an{1-x}\otimes\gpb{\frac{1}{1-\bar{y}}}.\\
\end{eqnarray*}
However, $(\an{x^{-1}-1}-\an{x})\otimes \bconst{k}=\an{x}\pf{1-x}\otimes \bconst{k}=\an{x}\otimes\pf{1}\bconst{k}=0$.
Thus $T(S_{x,y})=1\otimes Y$ where
\[
Y= \bconst{k}-\gpb{\bar{y}}+\gpb{\frac{1}{1-\bar{y}}}\in \qrpb{k}.
\]
Since $\bconst{k}=\gpb{\bar{y}}+\an{-1}\gpb{1-\bar{y}}$ in $\qrpb{k}$ we thus have 
\[
Y= \an{-1}\gpb{1-\bar{y}}+\gpb{\frac{1}{1-\bar{y}}}=\suss{1}{\frac{1}{1-\bar{y}}}=0 \mbox{ in }\qrpb{k}.
\]
\item[Subcase (b):] $v(x)<0$

We have $v(u)=-v(x)>0$ and $v(w)=v(1-x)=v(x)<0$ and $v(uw)=v(u)+v(w)=0$. 
Furthermore, since $1-x^{-1}\in U_1$, $\bar{uw}=1/(1-\bar{y}^{-1})$. Thus
\begin{eqnarray*}
T(S_{x,y})&=& -(1\otimes \bconst{k})-1\otimes\gpb{\bar{y}}+\an{x}\otimes\bconst{k}
-\an{x^{-1}-1}\otimes\gpb{\frac{1}{1-\bar{y}^{-1}}}-\an{1-x}\otimes\bconst{k}\\
&=& 1\otimes Y
\end{eqnarray*}
where 
\[
Y=-\bconst{k}-\gpb{\bar{y}}-\an{-1}\gpb{\frac{1}{1-\bar{y}^{-1}}} 
\]
since $(\an{x}-\an{1-x})\otimes\bconst{k}=-\an{x}\pf{x^{-1}-1}\otimes\bconst{k}=\an{x}\otimes\pf{-1}\bconst{k}=0$. 

But 
\begin{eqnarray*}
-\bconst{k}-\gpb{\bar{y}}-\an{-1}\gpb{\frac{1}{1-\bar{y}^{-1}}}&=&-
\bconst{k}-\gpb{\bar{y}}-\an{-1}\gpb{\frac{1}{1-\bar{y}^{-1}}}+\suss{1}{1-\bar{y}^{-1}}\\
&=& -\bconst{k}-\gpb{\bar{y}}+\gpb{1-\bar{y}^{-1}}\\
&=& -\bconst{k}-\gpb{\bar{y}}+\gpb{1-\bar{y}^{-1}}+\suss{1}{y}\\
&=& -\bconst{k}+\an{-1}\gpb{\bar{y}^{-1}}+\gpb{1-\bar{y}^{-1}}=0.
\end{eqnarray*}

\end{enumerate}
\item[Case (vii):] $x\in U\setminus U_1$ and $y\in U_1$.

We have $v(u)=0$ and $\bar{u}=\bar{x}^{-1}$. Furthermore, $v(w)=-v(1-y)<0$ and thus $v(uw)=v(w)<0$ also.
Thus
\begin{eqnarray*}
T(S_{x,y})&=&1\otimes\gpb{\bar{x}}+\an{x}\otimes\gpb{\bar{x}^{-1}}+\an{x^{-1}-1}\otimes\bconst{k}
-\an{1-x}\otimes\bconst{k}\\
&=& 1\otimes Z
\end{eqnarray*}
where $Z=\gpb{\bar{x}}+\an{\bar{x}}\gpb{\bar{x}^{-1}}+\an{\bar{x}^{-1}-1}\bconst{k}-\an{1-\bar{x}}\bconst{k}$. 

Thus 
\begin{eqnarray*}
Z&=&
\an{1-\bar{x}^{-1}}\suss{2}{\bar{x}^{-1}}+\suss{2}{\bar{x}^{-1}-1}-\suss{2}{1-\bar{x}}
\mbox{ by Corollary \ref{cor:bconst} (1)}\\
&=& 
\an{\bar{x}^{-1}-1}\suss{2}{\bar{x}}+\suss{2}{\bar{x}^{-1}-1}-\suss{2}{1-\bar{x}}
\mbox{ since $\suss{2}{\bar{x}^{-1}}=\an{-1}\suss{2}{\bar{x}}$}\\
\end{eqnarray*}
and this is $0$ by the cocycle property.
\item[Case (viii):] $x\in U_1$ and $y\in U\setminus U_1$

We have $u\in U$ and $\bar{u}=\bar{y}$, $v(w)=v(1-x)>0$ and $v(uw)=v(w)>0$. So
\begin{eqnarray*}
T(S_{x,y})&=& -(1\otimes\gpb{\bar{y}})+\an{x}\otimes\gpb{\bar{y}}-\an{x^{-1}-1}\otimes\bconst{k}+
\an{1-x}\otimes\bconst{k}\\
&=& -\an{1-x}\otimes\pf{\bar{x}}\bconst{k}=0 \mbox{ since } \bar{x}=1.
\end{eqnarray*}
\item[Case (ix):] $x,y\in U\setminus U_1$

In this case $\bar{x},\bar{y}\in k^\times\setminus \{ 1\}$ and 
\[
T(S_{x,y})=1\otimes S_{\bar{x},\bar{y}}=1\otimes 0.
\]

\end{enumerate}

\end{proof}

\begin{rem}
It is not hard to see that the specialization map $\spec{v}$ does not lift to a map 
with target $\indf{k}{F}{\rpb{k}}$. 

Indeed, if we make the calculation in case (vii) above carefully, then using 
Lemma \ref{lem:df} instead of Corollary \ref{cor:bconst} (1), we find that 
\[
T(S_{x,y})= 1\otimes \an{\bar{x}^{-1}-1}\suss{1}{\bar{x}}\in \indf{k}{F}{\ks{1}{k}}
\subset\indf{k}{F}{\rpb{k}}.
\]

Thus the homomorphism induced by $T$ will not induce a well-defined  homomorphism $\spec{v}$ 
on $\rpb{F}$ 
unless we  annihilate $\ks{1}{k}$ in the target. 
\end{rem}

\bibliography{Localfields}
\end{document}